\documentclass{article}

\usepackage{arxiv}

\usepackage[utf8]{inputenc} 
\usepackage[T1]{fontenc}    
\usepackage{hyperref}       
\usepackage{url}            
\usepackage{booktabs}       
\usepackage{nicefrac}       
\usepackage{microtype}      
\usepackage{xcolor}         

\usepackage{lipsum}		
\usepackage{graphicx}
\usepackage{natbib}
\usepackage{doi}

\usepackage{caption}
\usepackage{subcaption}
\usepackage{url}

\usepackage{math_commands}

\title{Dictionary learning for kernel EDMD}
\author{Erik~Lien~Bolager~\(^{1}\)\hspace{0.4cm} Boumediene~Hamzi~\(^{2,3}\) \hspace{0.4cm} Houman~Owhadi~\(^{2}\) \hspace{0.4cm} 
Ioannis G. Kevrekidis~\(^{4}\)\\  
\textbf{Felix~Dietrich}~\(^{1}\)\thanks{Corresponding author, \texttt{felix.dietrich@tum.de}}\\
	\(^1\)School of Computation, Information and Technology,\\Munich Data Science Institute,  Munich Center for Machine Learning,\\
	Technical University of Munich, Germany \\
    \(^2\)Department of Computing and Mathematical Sciences, Caltech, USA\\
    \(^3\)The Alan Turing Institute, United Kingdom\\
    \(^4\) Departments of Chemical and Biomolecular Engineering\\ and of Applied Mathematics and Statistics, Johns Hopkins University, USA
} 
\date{}

\begin{document}

\maketitle

\begin{abstract}
Studying nonlinear dynamical systems through their state space behavior can be challenging, and one possible alternative is to analyze them via their associated Koopman operator. This turns the nonlinear problem into a linear, infinite-dimensional one. To approximate the operator in finite dimensions, extended dynamic mode decomposition (EDMD) is a commonly used algorithm. It requires a finite list of functionals and a set of snapshots from the system to compute an approximation of the operator and its corresponding spectrum. Instead of choosing the list of functionals directly, it can be implicitly defined via kernels, a method known as kernel extended dynamic mode decomposition (kEDMD). However, one still needs to define the kernel and choose its parameter values.

In this paper, we aim to streamline this process by extending dictionary learning for EDMD to kernel learning in kEDMD. By simplifying kEDMD, we show how to perform gradient-based optimization over the learnable kernel parameters and demonstrate that this method leads to useful kernels for the original kEDMD. The focus of our work is a method that takes a weighted list of kernels with randomly initialized values as input and outputs a list of kernels and parameter values suitable for approximating the Koopman operator of the underlying system. We demonstrate that unimportant kernels can be removed from the list by analyzing the weights in the weighted sum. We evaluate the method across several experiments, including the Duffing oscillator and the Kuramoto-Sivashinsky equation, showcasing the method's different strengths.
\end{abstract}

\section{Introduction}
A key component of a dynamical system is its trajectories described by the system's evolution operator, and notions such as stability, periodicity, and other geometrical properties are of interest. Traditionally, this has been studied through the system's phase space. Another approach is to reframe it to an operator view of the system \citep{mezic-2004, mezic-2005, mezic-2023}, namely, the Koopman operator framework \citep{koopman-1931}. Here, the evolution of the system is studied in a functional space rather than in the phase space. Elements in this space are known as observables, and the time evolution of observables relevant for the specific system can then be studied. One example of an observable is the position of a state along one axis, which allows us to retrieve information from the phase space using observables. The reason for looking at the system through its corresponding Koopman operator is that the operator is linear, regardless of whether the underlying system is linear or not~\citep{budisic-2012}. Linearity of the operator also means that one can study the system through the spectrum and the corresponding eigenfunctions. However, in most cases, the required space of observables is infinite-dimensional. This then requires finite-dimensional approximations, for which we may use well-studied tools from numerical linear algebra due to the linearity of the operator. Studying systems through the Koopman operator lens has been used for prediction and system control, with applications ranging from oceanography \citep{giannakis-2015}, fluid dynamics \citep{mezic-2013, sharma-2016}, molecular kinetics \citep{nuske-2014, nuske-2016}, optimal control \citep{korda-2018b, mamakoukas-2019, bevanda-2022a} and beyond.

Several numerical methods have been proposed for approximating the Koopman operator \cite{rowley-2009, schmid-2010, williams-2015a, giannakis-2015, li-2017, klus-2018, colbrook-2021,schmid-2022}, and one data-driven approach is the extended dynamic mode decomposition (EDMD) \cite{williams-2015a}. As its name suggests, EDMD is an extension of dynamic mode decomposition (DMD) \cite{schmid-2010} using a larger list of (crucially, nonlinear) observables of the state. This extension allows the numerical approximation to converge to the operator under certain conditions, and error bounds exist \cite{mezic-2020, nuske-2023a, llamazares-elias-2024}. However, truncation to finite dimensionality also introduces challenges, such as spectral pollution, and methods such as residual DMD have been proposed to address them \cite{colbrook-2021, colbrook-2024}. The key components of EDMD are the dictionary, which is a finite collection of observables, and the set of snapshots from the time evolution of the system. We first evaluate the dictionary on the snapshots and then use the resulting matrices to approximate the Koopman operator, its spectrum, and eigenfunctions. Assuming one has a sufficient number of snapshots from the underlying dynamical system, the choice of functionals to include in the dictionary is the crucial part \cite{kutz-2018}. A good dictionary captures the properties of the system one is interested in to a satisfactory degree of precision, without being too large. The more observables included in the dictionary, the greater the computational effort is required, and this can also lead to numerical issues. To find a suitable dictionary, dictionary learning has been proposed \cite{li-2017}, and with several approaches in similar spirit using neural networks \cite{li-2017, otto-2019, constante-amoresricardo-2024, jin-2024, bolager-2026}, reservoir computing \cite{gulina-2021, bollt-2021a}, and random features \cite{salam-2022} have followed. In dictionary learning, one initializes a dictionary, typically using neurons from a neural network, and then applies gradient-based optimization to improve the dictionary. However, to make it computationally feasible, one still needs to limit the size of the dictionary. If the system requires many more observables than snapshots, implicitly defining the dictionary via a kernel may be a solution \cite{williams-2015b}.

Kernels play a major role in machine learning and can often be applied through the so-called kernel trick to replace inner products \citep{hastie-2009, owhadi-2019,hamzi-2024}. It also plays a pivotal role in techniques such as kernel regression and Gaussian process (GP) \citep{rasmussen-2005,yang-2021c, lee-2025b} and has deep connections to neural networks \citep{jacot-2018,lee-2018, avidan-2024}. Kernels have also been applied to dynamical systems in general \citep{bouvrie-2017, hamzi-2021}, to approximate the underlying manifold of the system 
\citep{haasdonk-2018,bitteracher-2019, haasdonk-2021}, for uncertainty propagation \citep{hou-2024}, and to the controlled system setting \citep{hamzi-2019}. One of the motivations behind such methods is that one maps the state to a high-dimensional (possibly infinite-dimensional) space, where one can treat the system in a linear fashion \citep{bouvrie-2010, bouvrie-2012}. This is very similar to the motivation behind the Koopman operator framework. In terms of the Koopman operator, works with kernels include, but are not limited to, approximating the Koopman operator with a matrix \citep{williams-2015b, klus-2020}, approximating the eigenfunctions of the operator \citep{lee-2025}, forecasting through kernel analog techniques \citep{alexander-2020}, studying the Koopman operator under dissipative systems \cite{mezic-2019}, and for model predictive control \citep{bold-2025}. Benefits include an inherent connection to reproducing kernel Hilbert spaces (RKHS) \citep{boulle-2025}, having far fewer parameters than typical neural networks, and better interpretability than neural networks. In \citet{williams-2015b}, they propose to implicitly define the dictionary in the EDMD algorithm using kernels, known as kernel EDMD (kEDMD). The method has also been studied for convergence \citep{philipp-2025}, and techniques developed for EDMD have been extended to kEDMD \citep{colbrook-2021}. The upside of using kernels is that the size of the Koopman approximation is bounded by the number of snapshots, but with the downside that one cannot evaluate the dictionary directly. The choice of kernels has so far relied on intuition and manual parameter tuning. One reason is that we cannot simply extend dictionary learning for EDMD to kernels, since we cannot evaluate the dictionary. In this paper, we aim to tackle this problem. We show how, in theory, we can perform dictionary learning for the original kEDMD method. We then simplify the kEDMD method, showing that this simplification is equivalent to the original method, and then develop a method to learn the kernel parameters for the simplified version. The simplification allows for many more loss functions and is closer to the original dictionary learning scheme. This method can then take in a weighted sum of well-known kernels, learn the parameters given a suitable loss function, and also be used to find a sparse representation of the kernels; an idea similar to automatic relevance determination (ARD) \citep{williams-1995, mackay-1996}.

The rest of this paper is outlined as follows: in \Cref{sec:preliminaries} we introduce the Koopman operator, the different kinds of EDMD methods, and dictionary learning for EDMD. In \Cref{sec:method}, we show how we can perform dictionary learning for kEDMD, as well as how we can simplify the original kEDMD method and develop a dictionary learning for the new version. In \Cref{sec:experiments}, we perform several experiments using dictionary learning for the simplified version, and in \Cref{sec:conclusion}, we provide a conclusion and share a few thoughts on future work.

\section{Preliminaries}\label{sec:preliminaries}
We let \(T \colon \mathcal{X} \rightarrow  \mathcal{X}\) be the evolution operator of a discrete dynamical system, with \(\mathcal{X} \subseteq \R^d\), i.e., \(\vx_{t+1} = T(\vx_t)\) for \(t \in \mathbb{N}_{\geq 0}\). In addition, we let \(\mathcal{F}\) be a functional space over \(\mathcal{X}\) where we usually assume \(\mathcal{F}\coloneq \{f\colon \mathcal{X}\rightarrow \R \colon \lVert f\rVert_{L_2(\mathcal{X}, \mu)} < \infty\}\), where \(\lVert \cdot \rVert_{L_2(\mathcal{X}, \mu)}\) is the canonical \(L_2\) norm with a suitable measure \(\mu\). The Koopman operator \(\mathcal{K}_T\colon \mathcal{F} \rightarrow \mathcal{F}\), also known as the composition operator, is then defined as 
\begin{align*}
    [\mathcal{K}_T f](\vx) = f(T(\vx)), \quad f \in \mathcal{F}, \vx \in \mathcal{X}. 
\end{align*}
For convenience’s sake, we opt to drop the subscript \(T\) in \(\mathcal{K}_T\). In essence, through the Koopman operator, we change the viewpoint we analyze the dynamical system from; instead of considering the possibly nonlinear evolution of \(T\) in a finite-dimensional space \(\mathcal{X}\), we analyze how functionals evolve by applying the linear operator \(\mathcal{K}\) on an infinite-dimensional space. The trade-off by gaining linearity while losing finite dimensionality, allows us to study the dynamical system through the spectrum and eigenfunctions of \(\mathcal{K}\). For example, let \(\{\varphi_j\}_{j=1}^{\infty}\) be eigenfunctions of \(\mathcal{K}\), and assume \(f \in \spa \{\varphi_j\}_{j=1}^{\infty}\), then
\begin{align*}
    f(\vx_t) = [\mathcal{K}^t f](\vx_0) = \sum_{j=1}^{\infty} c_j \lambda_j^t \varphi_j(\vx_0),
\end{align*}
where \(\lambda_j\) is the eigenvalue associated to the eigenfunction \(\varphi_j\), and \(\{c_j\}_{j=1}^{\infty}\) are the coefficients such that \(f = \sum_{j=1}^{\infty} c_j \varphi_j\). These coefficients are usually referred to as the Koopman modes. One common functional often considered is \(f_i(\vx) = \vx_i\), i.e., a mapping from the vector \(\vx\) to the \(i\)'th component of the vector. If \(f_i \in \mathcal{F}\) for \(i=1,\dots, d\), we may stack the functionals \(f= [f_1, \dots, f_d]\tran\) and apply the Koopman operator
\begin{align*}
   \vx_t =  f(\vx_t) = \sum_{j=1}^{\infty} \vc_j \lambda_j^t \varphi_j(\vx_0).
\end{align*}
For further reading on the Koopman operator and its application to dynamical systems, see \citet{budisic-2012}.

\subsection{Extended dynamic mode decomposition}
Approximating the Koopman operator in a data-driven manner can be achieved using extended dynamic mode decomposition (EDMD). We give a brief introduction here, as it will be helpful when we rewrite kernel-based EDMD later on. We first assume we have access to snapshots from the dynamical system, 
\begin{align*}
    \mathcal{D} = \{(\vx_1, \vy_1), (\vx_2, \vy_2), \dots, (\vx_N, \vy_N), (\vx_{N+1}, \vy_{N+1})\},
\end{align*}
and define the two matrices
\begin{align*}
    X = [\vx_1, \vx_2, \dots, \vx_N]\tran \in \mathbb{R}^{N\times d}, \quad &\text{and}\quad Y= [\vy_1, \vy_2, \dots, \vy_N]\tran \in \mathbb{R}^{N\times d},
\end{align*} 
where \(\vy_n = T(\vx_n)\). We then select a set of observables from the space \(\mathcal{F}\), which we call a dictionary \(\Psi = \{\psi_1, \psi_2, \dots, \psi_M\} \in \F^{M}\), with
\begin{align*}
    \psi_i \colon \mathcal{X} \rightarrow \R,\quad \text{for } i=1,2,\dots, M.
\end{align*} 
When evaluating the dictionary on our dataset, we denote
\begin{align*}
    \Psi(\vx) = [\psi_1(\vx), \psi_2(\vx), \dots, \psi_M(\vx)]\tran \in \R^M, \quad \Psi(X) = [\Psi(\vx_1), \Psi(\vx_2), \dots, \Psi(\vx_N)] \in \R^{M \times N}.
\end{align*}
We can then obtain an approximation of the Koopman operator \(K\in \R^{M\times M}\), by minimizing 
\begin{align}\label{eq:opt_full_K}
     K&= \underset{\tilde{K} \in \R^{M\times M}}{\arg\min} \lVert \Psi(Y) - \tilde{K}\Psi(X) \rVert^2_F,
\end{align}
where \(\lVert \cdot \rVert_F\) is the Frobenius norm. A closed-form solution can be written as \(K = \Psi(Y) \Psi(X)^+\), where \(^+\) is the (pseudo)inverse. From this approximation, one can obtain an approximation of eigenfunctions of the Koopman operator given by
\begin{align*}
    \phi_j = \vw_j \Psi,
\end{align*}
where \(\vw_j\) is the \(j\)th left eigenvector of \(K\). We denote all the eigenfunction approximations as 
\begin{align*}
    \Phi = W \Psi,
\end{align*}
where each row in \(W\) is a left eigenvector. Finally, once the functional(s) are defined, e.g., \(f=[f_1, \dots, f_d]\tran\) where \(f_i(\vx) = \vx_i\), we can approximate the modes. Let \(d_o\) be the output dimension of \(f\) and \(\Lambda\) be a diagonal matrix with the eigenvalues of \(K\), we then find the modes \(C_m\) by
\begin{align}\label{eq:koopman_modes}
    C_m = \argmin_{\Tilde C \in \C^{d_o \times M}} \lVert f(Y) - \Tilde C \Lambda \Phi(X) \rVert_F^2.
\end{align}
When \(W\) is full-rank, one can show that the resulting modes are the right eigenvectors, but numerically we rather solve the least-squares problem with suitable regularization. 

\subsubsection{Residuals and spectral pollution}\label{sec:res_spec}
A well-known problem when approximating the spectrum of infinite-dimensional operators is spectral pollution. That is, the approximation produces eigenvalues that are not related to the operator, but are rather artifacts due to the finite-dimensional discretization of the operator \citep{lewin-2009,colbrook-2019}. This also holds true for EDMD, and methods to quantify the issue and attempt to remove the pollution have been proposed. In particular, \citet{colbrook-2021} compute the residuals for each eigenvalue and eigenfunction, which should be low if the approximated eigenvalue is close to a true eigenvalue. Let \((\lambda, \phi)\) be an approximate eigenvalue and eigenfunction, with \(\vw\) being the corresponding left eigenvector; then we may compute
\begin{align*}
    \text{res}^2(\lambda, \phi) &= \frac{\lvert\int_{\mathcal{X}} [\mathcal{K}\phi](\vx) - \lambda \phi(\vx)\rvert^2 d\mu(\vx)}{\int_{\mathcal{X}} \lvert \phi(\vx)\rvert^2d\mu(\vx)}\\ 
    &\approx \frac{\vw \left(\Psi(Y) \Psi(Y)\tran - \lambda \Psi(Y)\Psi(X)\tran - \overline{\lambda} \Psi(X)
    \Psi(Y)\tran + \lvert\lambda\rvert^2 \Psi(X)\Psi(X)\tran\right) \vw^*}{\vw \left(\Psi(X)\Psi(X)\tran \right)\vw^*}.
\end{align*}
Here the left eigenvectors are usually found through solving the generalized eigenvalue problem, instead of computing the Koopman approximation first. The residuals can then be used to compute the eigenpairs with less spectral pollution by setting an appropriate threshold. In certain systems this can be quite beneficial \citep{colbrook-2021}.

Apart from spectral pollution, two more challenges to the EDMD algorithm are (a) one needs to choose a suitable dictionary, which can partially be resolved by learning the dictionary \citep{li-2017}, and (b) where \(M \gg N\), the computational cost of solving for \(K\) becomes large; this can be resolved by kernel extended dynamic mode decomposition (kEDMD) \citep{williams-2015b}.

\subsection{Kernel extended dynamic mode decomposition}
To alleviate the computational burden imposed by a large \(M\) one can, instead of defining the dictionary in EDMD explicitly, define it implicitly through a kernel \citep{williams-2015b}. For simplicity, let us assume that the function space induced implicitly by a kernel \(g\colon \mathcal{X} \times \mathcal{X} \rightarrow \R\) is of finite size \(M \gg N\). The function space can then be written as a span of \(M\) functions which we may refer to as the dictionary \(\Psi\). To derive kEDMD, we start by  applying SVD to \(\Psi(X) = \Tilde Q \Tilde \Sigma \Tilde Z\tran\), and then reducing it to \( \Psi_P = Q\Sigma Z\tran\), where \(Q\in \R^{M\times N}\) is now a semi-orthogonal matrix and \(\Sigma, Z \in \R^{N\times N}\). We then define an orthogonal projection \(P = QQ\tran\), which projects onto the space spanned by the columns of Q. After this reduction of the original dictionary, we may define the Koopman approximation similarly to \Cref{eq:opt_full_K}, i.e.,
\begin{align*}
   K_P&= \underset{\tilde{K} \in \R^{M\times M}}{\arg\min} \lVert P\Psi(Y) - (P\tilde K)P\Psi(X)\rVert^2_F.
\end{align*}
Considering the solution \(K_P = (P\Psi(Y))(P\Psi(X))^+\) one can show that
\begin{align}\label{eq:def_khat}
    K_P = (P\Psi(Y))(P\Psi(X))^+ = Q(Z\Sigma^+)\tran ((P\Psi(X))\tran (P\Psi(Y))) (Z\Sigma^+) Q\tran = Q \Hat K Q\tran.
\end{align}
The key part of the equation above is that \(\Hat K \in \R^{N\times N}\), and if one only needs to compute \(\Hat K\), the high dimensionality of \(M\) is avoided. As it turns out, to capture the full eigenspace of \(K_P\), one only needs \(\Hat K\). This was shown by \citet{williams-2015b}, and is summarized as the following proposition.
\begin{proposition}\label{prop:khat_spectrum}[\citet{williams-2015b}]
        The following points hold for the relationship between the eigenvectors of \(K_P\) and \(\hat K\), with \(\lambda\) being the corresponding eigenvalue,
    \begin{itemize}
        \item If \(\vw\) is a left eigenvector of \(K_P\), then \(\hat \vw = \vw Q\) is a left eigenvector of \(\hat K\).
        \item If \(\hat \vw\) is a left eigenvector of \(\hat K\), then \(\vw = \hat \vw Q\tran \) is a left eigenvector of \(K_P\). 
        \item If \(\vv\) is a right eigenvector of \(K_P\), then \(\hat \vv = Q\tran \vv\) is a right eigenvector of \(\hat K\).
        \item If \(\hat \vv\) is a right eigenvector of \(\hat K\), then \(\vv = Q\vv\) is a right eigenvector of \(K_P\).
    \end{itemize}
\end{proposition}
It is worth noting that the result really follows from the fact that we apply reduced SVD to the original dictionary \(\Psi\) to achieve \(\Psi_P\), and from this reduction we only have \(N\) eigenvalues/eigenvector pairs to represent the same information as we obtain from \(K_P\).

To introduce kernels to the method described so far, we note that, to compute \(\Hat K\), we need to compute \(Z\), \(\Sigma\), and \(((P\Psi(X))\tran (P\Psi(Y)))\). For the latter, as the matrix are produced by inner products between the dictionary elements induced by the kernel, the resulting matrix is the kernel matrix between elements in \(X\) and elements in \(Y\), namely \(g(Y,X) =((P\Psi(X))\tran (P\Psi(Y))))\). For the former two, note that 
\begin{align*}
    g(X,X) =((P\Psi(X))\tran (P\Psi(X)))) = Z\Sigma Q\tran Q \Sigma Z\tran = Z\Sigma^2 Z\tran,
\end{align*}
which means we may compute \((Z, \Sigma)\) through the eigendecomposition of the Gram matrix. Hence, we only need to define the kernel to compute \(\Hat K \) and its spectrum. However, this also means we do not have access to the dictionary and cannot evaluate it directly, which is even clearer when choosing a kernel that induces an infinite-dimensional function space. We also do not have access to the approximate eigenfunctions \(P\Phi = WP\Psi\), as this also requires the dictionary \(\Psi\). Fortunately, we can evaluate the approximate eigenfunctions pointwise. Evaluating at points in \(X\) can be done by
\begin{align*}
   P\Phi(X) = W P \Psi(X)) = WQQ\tran Q \Sigma Z\tran = \Hat W \Sigma Z\tran,
\end{align*}
where \(\Hat W\) is the left eigenvectors of \(\Hat K\), and we apply \Cref{prop:khat_spectrum} for the equality \(WQ = \Hat W\). For evaluating the eigenfunctions at new points \(\vx \in \mathcal{X}\), we have
\begin{align*}
    P\Phi(\vx) = WP\Psi(\vx) = \Hat W \Sigma^+ Z\tran (P\Psi(X))\tran(P\Psi(\vx)) = \Hat W \Sigma^+ Z\tran g(\vx, X).
\end{align*}
The apparent trade-off between access to the dictionary and avoiding computational cost involving \(M\), can be beneficial if \(M \gg N\) and the goal is to study the spectrum or predict trajectories. Choosing a kernel can also be easier if one has prior information to the relationship between points \(\mathcal{X}\), and analytically the connection to reproducing kernel Hilbert spaces and Gaussian processes among others can be of interest. However, one still needs to choose a suitable kernel.

\subsection{Dictionary learning for EDMD}\label{sec:dl_edmd}
To choose a suitable dictionary of a given size \(M\) in the EDMD algorithm, \citet{li-2017} introduce dictionary learning. In particular, by parameterizing the dictionary, in their case a neural network with \(M\) outputs, they optimize the parameters using gradient optimization. More specifically, denoting the parameterized dictionary as \(\Psi_{\theta}\), where \(\theta\) are the weights and biases of the neural network, one can alternate between approximating the Koopman operator and using the approximation to optimize the parameters. That is, iterating between 
\begin{enumerate}
    \item holding the values of \(\theta\) fixed and computing \(K\) following \Cref{eq:opt_full_K},
    \item holding \(K\) fixed and performing one step of gradient descent on \(\theta\) using a suitable loss function.
\end{enumerate} 
For a suitable loss function, \citet{li-2017} propose to use a dictionary loss function, namely
\begin{align*}
    \mathcal{L}_{dict}(\theta) = \lVert \Psi_{\theta}(Y) - K \Psi_{\theta}(X)\rVert_F^2.
\end{align*}
In addition to the dictionary loss function, they added an \(L_2\) regularization of \(\theta\) for stability. A slight issue with the proposed loss is that minimizing it allows trivial solutions, such as \(\Phi(\vx) = 0\). To prevent this, one also includes a few non-parameterized functions, such as the identity function, in the dictionary. The approach can then be summarized as in \Cref{alg:dl_edmd}.

\begin{minipage}{1.0\linewidth}
    \begin{algorithm}[H]
        \caption{Dictionary learning for EDMD, using a parameterized dictionary \(\Psi_{\theta}\). As input it takes snapshots from the dynamical system \((X,Y)\), initial parameter value \(\theta_0\), learning rate \(\eta >0\), regularization constant \(\beta>0\), and the number of epochs. We denote the identity function as \(f_{\text{Id}}\), i.e., \(f_{\text{Id}}(\vx)= \vx\).}\label{alg:dl_edmd}
        \begin{algorithmic}
            \Procedure{DL-EDMD}{$X$, $Y$, $\theta_0$, $\eta$, $\beta$, $N_{\text{epochs}}$}
            \State $\theta \leftarrow \theta_0$
            \State $\Bar{\Psi}_{\theta} \leftarrow \Psi_{\theta} \cup f_{\text{Id}}$
            \For{$e=1,2,\dots, N_{\text{epochs}}$}
                \State $K \leftarrow \Bar{\Psi}_{\theta}(Y) \Bar{\Psi}_{\theta}(X)^+$
                \State $\theta = \theta - \eta \frac{\partial}{\partial \theta}(\lVert \Bar{\Psi}_{\theta}(Y)  - K \Bar{\Psi}_{\theta}(X)\rVert_F^2 + \beta\lVert \theta \rVert^2_2$)
            \EndFor
            \State Return $\theta$
            \EndProcedure
    \end{algorithmic}
    \end{algorithm}
\end{minipage}

The method has successfully been applied to dynamical systems, in particular by using neural networks to parameterize the dictionary. Extensions have been proposed to avoid the sequential nature of alternating between computing \(K\) and performing the gradient step \citep{constante-amoresricardo-2024}. The question then becomes whether something similar can be applied for parameterized kernels in kEDMD.

\section{Method}\label{sec:method}
A major issue when it comes to dictionary learning for kEDMD (kEDMD-DL) is the fact that we do not have access to the dictionary directly, and hence defining a loss function similar to the dictionary loss function in \Cref{sec:dl_edmd} is difficult. We will first show how this in theory can be performed, before we show how we can simplify kEDMD to allow for a version of kEDMD-DL closer to the dictionary learning in \Cref{sec:dl_edmd}.

\subsection{Loss functions for kEDMD-DL}\label{sec:loss_trunc_kedmd}
To perform dictionary learning for kEDMD, we assume the kernel \(g_{\theta}\) is parameterized by \(\theta\), and the dictionary induced implicitly is denoted as \(\Psi_{\theta}\). To be able to compute the dictionary loss similarly to dictionary learning for EDMD, we need to be able to evaluate the following loss function
\begin{align}\label{eq:kedmd_dictloss}
    \mathcal{L}_{dict}(\theta) = \lVert P\Psi_{\theta}(Y) - K_P(P\Psi_{\theta}(X)) \rVert^2_F.
\end{align}
We use \(K_P\) here instead of \(\Hat K\) simply because the dictionary implicitly defined through the kernel has \(M\) elements. To show how we can evaluate the loss function, we need a few properties of \(K_P\) and \(\Hat K\). First, remember that \(P\) is idempotent and equal to its pseudoinverse, and we therefore have 
\begin{align}\label{eq:pkp}
    PKP = P\Psi_{\theta}(Y)\Psi_{\theta}(X)^+P = P\Psi_{\theta}(Y)(P\Psi_{\theta}(X))^+ = K_P.
\end{align}
If one assumes full-rank of \(\Hat W\), i.e., the left eigenvectors of \(\Hat K\), we have
\begin{align}\label{eq:pvwp}
    PVWP = QQ\tran V W QQ\tran = Q \Hat V \Hat WQ\tran = P
\end{align}
where we apply \Cref{prop:khat_spectrum} and the fact that with full-rank the inverse of \(\Hat W\) is \(\Hat V\). Using \Cref{eq:pkp} and \Cref{eq:pvwp}, we may then show 
\begin{align*}
    PKP\Psi_{\theta}(X) =  (PK) P (P\Psi_{\theta}(X)) =  (PK) PVWP  (P\Psi_{\theta}(X)) = P(PKP) V (WP\Psi_{\theta}(X)) = PV\Lambda  P\Phi_{\theta}(X)
\end{align*}
and 
\begin{align*}
    P\Psi(Y) = PP\Psi(Y) = PVWP\Psi(Y) = PVP\Phi(Y).
\end{align*}
Here, \(\Phi_{\theta}\) is the approximate eigenfunctions using the dictionary \(\Psi_{\theta}\). Applying \Cref{prop:khat_spectrum} to these two equalities, we end up with the following loss function:
\begin{align*}
   \mathcal{L}_{dict}(\theta) &= \lVert PVP\Phi_{\theta}(Y) - PV\Lambda P\Phi_{\theta}(X)\rVert^2_F\\ 
   &= \lVert Q\Hat V P\Phi_{\theta}(Y) - Q\Hat V \Lambda P\Phi_{\theta}(X)\rVert_F^2 = \lVert \Hat V P\Phi_{\theta}(Y) - \Hat V \Lambda P\Phi_{\theta}(X) \rVert_F^2,
\end{align*}
where the final equality follows from the fact that \(Q\) is an isometry in the Euclidean space when applied from the left \citep{abadir-2005}. As we already know how to evaluate the eigenfunctions, we end up with the following loss expression:
\begin{align*}
    \mathcal{L}_{dict}(\theta) = \lVert P\Psi_{\theta}(Y) - K_P P\Psi_{\theta}(X) \rVert_F^2 = \lVert \Hat V \Hat W \Sigma^+ Z\tran g_{\theta}(Y,X) - \Hat V \Lambda \Hat W \Sigma Z\tran\rVert_F^2,
\end{align*}
where one can simplify even further due to the full-rank assumption if one wishes to. If the loss function is not evaluated at the same snapshots as used to compute \(\Hat K\), say we rather use \(\Tilde X, \Tilde Y \in \R^{ d\times \Tilde M}\), we modify the expression above slightly to 
\begin{align*}
    \mathcal{L}_{dict}(\theta) = \lVert \Hat V \Hat W \Sigma^+ Z\tran g_{\theta}(\Tilde Y,X) - \Hat V \Lambda \Hat W \Sigma^+ Z\tran g_{\theta}(\Tilde X, X)\rVert_F^2.
\end{align*}
The fact that we compute \(\Hat K\) and then perform the gradient step means we only need to compute the gradients w.r.t. the kernel matrices. However, we are still facing two issues. The first one is that we  assume the full-rank of \(\Hat V\) and \(\Hat W\), which is not guaranteed. The second issue is that the dictionary loss function can be minimized by trivial solutions. The former may cause instabilities when the assumption is not satisfied, the latter is harder to fix due to the fact that we cannot add additional functions to the dictionary, as was done in dictionary learning for EDMD. The addition to the loss function such as adding the term \(\lVert X\tran - C_p\Psi_{\theta}(X)\rVert_F^2\), where \(C_p\) is a projection down to the state space, may be a possible fix. Another one is to rather use an eigen loss,
\begin{align}\label{eq:kedmd_eigloss}
    \mathcal{L}_{eig}(\theta) = \lVert P\Phi_{\theta}(Y) - \Lambda P\Phi_{\theta}(X)\rVert_F^2,
\end{align}
or an eigen prediction loss 
\begin{align}\label{eq:kedmd_eigpredloss}
    \mathcal{L}_{eig-pred}(\theta) = \lVert Y\tran - C_m\Lambda P\Phi_{\theta}(X)\rVert_F^2,
\end{align}
where \(C_m\) are the Koopman modes. We find, however, that using kEDMD in its original form is still too limiting when it comes to dictionary learning, and indeed harder to work with computationally, and therefore proceed to simplify the approach.

\subsection{Simplifying kEDMD}\label{sec:skedmd}
When developing the kEDMD, taking the approach of starting out with truncated SVD may not be entirely necessary. Instead, we start by defining a dictionary as one would in EDMD, and rather involve the kernel directly. That is, the dictionary, given a kernel \(g_{\theta}\), and snapshots \(X = [\vx_1, \dots, \vx_N]\), is defined as 
\begin{align}\label{eq:skedmd_dict}
    \Psi^{sk}_{\theta} = \{g_{\theta}(\cdot, \vx_1), g_{\theta}(\cdot, \vx_2), \dots, g_{\theta}(\cdot, \vx_N)\}.
\end{align}
Here the dictionary are the kernels where the second input is set to be each of the \(N\) snapshots given in \(X\). The size of the dictionary is therefore \(M = N\). Applying the standard EDMD now gives us 
\begin{align}\label{eq:k_simplified}
    K = \Psi^{sk}_{\theta}(Y) \Psi^{sk}_{\theta}(X)^+ = g_{\theta}(Y,X) g_{\theta}(X,X)^+,
\end{align}
where the dimensionality of \(K\) is \(N\times N\), and therefore similar to computing \(\Hat K\). Even more so, the two matrices one needs is the kernel matrix evaluated at \(Y,X\) and \(X,X\), which are the same two matrices one needs when computing \(\Hat K\), but for the latter one uses the eigendecomposition of \(g_{\theta}(X,X)\) instead of the inverse. This relationship, as far as we are aware, was first identified in \citet{gonzalez-2021}, but not explored further. It has since then been independently from this paper explored further in work such as \citet{boulle-2025}. We will now present our results showing the clear connection between the two versions of kernel EDMD, and also how we may map between these two seamlessly.

We first start by denoting \(\Hat K\) as \(K_{tr}\) (\textit{tr} for truncated SVD) and \(K\) defined in \Cref{eq:k_simplified} by \(K_{sk}\). We now show explicitly the connection between the two approaches, and also the resulting eigenvalues and eigenvectors similarly to \Cref{prop:khat_spectrum}.
\begin{proposition}\label{prop:simp_kedmd}
    Let \(Z\) and \(\Sigma\) be defined as in \Cref{eq:def_khat}, and assume \(g_{\theta}(X,X)\) has full-rank. Then there exists a linear map between \(K_{tr}\) and \(K_{sk}\), namely,
    \begin{align*}
        K_{tr} = (Z \Sigma)^+ K_{sk} Z\Sigma.
    \end{align*}
    We then have the following relations for the eigenvalue/eigenvector pairs of \(K_{tr}\) and \(K_{sk}\),
    \begin{itemize}
        \item If \((\lambda, \vw_{tr}, \vv_{tr})\) are the eigenvalue and corresponding left/right eigenvectors for \(K_{tr}\), then \((\lambda, \vw_{tr} \Sigma^+ Z\tran, Z\Sigma \vv_{tr})\) are the eigenvalue and corresponding left/right eigenvectors for \(K_{sk}\).
        \item If \((\lambda, \vw_{sk}, \vv_{sk})\) are the eigenvalue and corresponding left/right eigenvectors for \(K_{sk}\), then \((\lambda, \vw_{sk} Z\Sigma, \Sigma^+Z\tran \vv_{sk})\) are the eigenvalue and corresponding left/right eigenvectors for \(K_{tr}\).
    \end{itemize}
\end{proposition}
\begin{proof}
   First recall that \(g_{\theta}(X,X) = Z\Sigma^2 Z\tran\). We then have
   \begin{align*}
        K_{sk} &= g_{\theta}(Y,X) Z \Sigma^+ \Sigma^+ Z\tran = g_{\theta}(Y,X) (Z\Sigma^+) (Z \Sigma^+)\tran,
   \end{align*} 
   and by recalling from \Cref{eq:def_khat} that \(K_{tr} = (Z\Sigma^+)\tran g_{\theta}(Y,X) (Z\Sigma^+)\), we find
   \begin{align*}
        K_{sk} &=  (Z \Sigma) (\Sigma^+Z\tran) g_{\theta}(Y,X) (Z\Sigma^+) (Z \Sigma^+)\tran = (Z\Sigma) K_{tr} (Z\Sigma^+)\tran = (Z \Sigma)K_{tr} (Z\Sigma)^+.
   \end{align*}
   Equivalently, we have
   \begin{align*}
        K_{tr} = (Z\Sigma)^+ K_{sk} (Z\Sigma).
   \end{align*}
   For completeness we also show the relations for the eigenvalue and eigenvectors. Assume \((\lambda, \vw_{tr}, \vv_{tr})\) are eigenvalue and corresponding eigenvectors of \(K_{tr}\). Then 
   \begin{align*}
        \vw_{sk}K_{sk} = \vw_{tr} \Sigma^+ Z\tran K_{sk} = \vw_{tr} (Z\Sigma)^+ K_{sk} (Z\Sigma) (Z\Sigma)^+ = \vw_{tr} K_{tr} (Z\Sigma)^+ = \lambda \vw_{tr}(Z\Sigma)^+ = \lambda \vw_{sk},
   \end{align*}
   and
   \begin{align*}
        K_{sk} \vv_{sk} = K_{sk} Z\Sigma \vv_{tr} = (Z\Sigma) (Z\Sigma)^+ K_{sk} (Z\Sigma) \vv_{tr} = \lambda (Z\Sigma) \vv_{tr} = \lambda \vv_{sk}.
   \end{align*}
   Similarly, assuming \((\lambda, \vw_{sk}, \vv_{sk})\) are eigenvalue and corresponding eigenvectors of \(K_{sk}\), we have
   \begin{align*}
        \vw_{tr}K_{tr} = \vw_{sk} Z\Sigma K_{tr} = \vw_{sk} (Z\Sigma) K_{tr} (Z\Sigma)^+ (Z\Sigma) = \lambda \vw_{sk} Z\Sigma = \lambda \vw_{tr}
   \end{align*}
   and
   \begin{align*}
      K_{tr} \vv_{tr} &= K_{tr} (Z\Sigma)^+ \vv_{sk} = (Z\Sigma)^+ K_{sk} (Z\Sigma) (Z\Sigma)^+ \vv_{sk} =  (Z\Sigma)^+ \lambda \vv_{sk} = \lambda \vv_{tr}.
   \end{align*}
\end{proof}
\begin{remark}
    In the result, we assume full-rank of the Gram matrix \(g_{\theta}(X,X)\), implying that the pseudo-inverse computation could be replaced by the inverse. If the assumption is not satisfied, similar results are possible by setting certain singular values in \(\Sigma\) to zero and simply truncating the space until the remaining vectors with nonzero singular values in \(Z\) has full-rank, and handle the pseudo-inverse accordingly.
 \end{remark}

This result highlights the similarity between the two methods and shows that we can easily map between them via either the Koopman approximations or the eigenspaces directly. It is therefore reasonable to believe that performing dictionary learning for one method, if successful, should yield a good kernel to use for the second method as well. Given that the available loss functions for dictionary learning using the original kEDMD are limited and rely on additional assumptions, we prefer to use the simplified kernel EDMD (skEDMD) when formulating kEDMD-DL. The simplified version can also be viewed as EDMD with a special restriction on the dictionary specified in \Cref{eq:skedmd_dict}, which means one can apply techniques designed for EDMD to skEDMD, and by \Cref{prop:simp_kedmd}, to kEDMD. However, the problem of trivial solutions when using a dictionary loss function, as in \Cref{sec:dl_edmd}, remains to be addressed.

\subsection{Dictionary learning for kEDMD}
By using skEDMD, we can evaluate the dictionary, making it much easier to create different loss functions. To avoid confusion, we denote losses defined in \Cref{eq:kedmd_dictloss}, \Cref{eq:kedmd_eigloss}, and \Cref{eq:kedmd_eigpredloss} as \(\mathcal{L}^{tr}_{dict}\), \(\mathcal{L}^{tr}_{eig}\), and \(\mathcal{L}^{tr}_{eig-pred}\) respectively. One natural choice when thinking of losses would be to follow \Cref{sec:dl_edmd} and use

\begin{align*}
    \mathcal{L}^{sk}_{dict}(\theta) = \lVert \Psi^{sk}_{\theta}(Y) - K_{sk} \Psi^{sk}_{\theta}(X)\rVert_F^2,
\end{align*}
which has the already mentioned issues. Since we cannot add additional functions to \(\Psi^{sk}_{\theta}\), as this would then not be equivalent to the original kEDMD, and \Cref{prop:simp_kedmd} would not hold anymore. A different solution would be to add the term \(\lVert X\tran - C_p\Psi^{sk}_{\theta}(X)\rVert^2_F\), where we compute the mapping down to the state space as 
\begin{align*}
    C_p = \argmin_{\Tilde C \in \C^{d \times N}} \lVert X\tran - \Tilde C\Psi^{sk}_{\theta}(X) \rVert_F^2.
\end{align*}
The matrix is, similarly to \(K_{sk}\), computed while holding the dictionary parameters fixed, and then holding the matrix fixed while updating \(\theta\). Instead of adding this as a second term, one can combine it with \(\mathcal{L}^{sk}_{dict}\) and define
\begin{align*}
    \mathcal{L}^{sk}_{pred}(\theta) = \lVert Y\tran - C_pK\Psi^{sk}_{\theta}(X)\rVert^2_F.
\end{align*}
This loss function is then optimizing for trajectory prediction; however, as we are interested in the spectrum and to use the eigenfunctions for computing trajectories, two other natural loss functions to consider are
\begin{align*}
    \mathcal{L}^{sk}_{eig}(\theta) = \lVert \Phi^{sk}_{\theta}(Y) - \Lambda \Phi^{sk}_{\theta}(X) \rVert_F^2 = \lVert W\Psi^{sk}_{\theta}(Y) - \Lambda  W\Psi^{sk}_{\theta}(X) \rVert_F^2  
\end{align*}
and
\begin{align*}
    \mathcal{L}^{sk}_{eig-pred}(\theta) = \lVert Y\tran - C_m\Lambda \Phi^{sk}_{\theta}(X) \rVert_F^2.
\end{align*}
We can then combine the loss functions into a weighted sum,
\begin{align*}
    \mathcal{L}^{sk}(\theta) = \alpha_1 \mathcal{L}^{sk}_{dict}(\theta) + \alpha_2 \mathcal{L}^{sk}_{pred}(\theta) + \alpha_3 \mathcal{L}^{sk}_{eig}(\theta) + \alpha_4 \mathcal{L}^{sk}_{eig-pred}(\theta),
\end{align*} 
with hyperparameters \(\{\alpha_i \in \mathbb{R}_{\geq 0}\}_{i=1}^4\). The addition of the loss functions involving the approximate eigenfunctions requires additional computational resources to compute the eigenvalues and eigenvectors. We also have observed experimentally that when using only \(\mathcal{L}^{sk}_{pred}\) when performing gradient steps leads to a similar decrease in loss for \(\mathcal{L}^{sk}_{eig-pred}\), and that in general \(\mathcal{L}^{sk}_{dict}\) and \(\mathcal{L}^{sk}_{eig}\) appear less helpful. This is to say, we do not rule out that there might exist systems where weighting the loss functions differently may lead to better learning, but we found so far that \(\mathcal{L}^{sk}_{pred}\) works well on its own.

\subsubsection{Choice of kernels}
Before one can apply dictionary learning to the kernel \(g_{\theta}\), one needs to choose how to parameterize the kernel. That is, we need to choose what kernels to include. Many kernels used in machine learning are available. Although one may choose a single kernel, such as the radial basis function (RBF) kernel, as we do in \Cref{sec:exp1}, we often do not know exactly which kernels will be useful and which will not, so we opt for a weighted sum of kernels. Let \(M_g \in \mathbb{N}_{>0}\) and \(\{g_{\theta_i}\}^{M_g}_{i=1}\) be kernels parameterized by \(\theta_i\), respectively, where we refer to the collection of \(\theta_i\)'s as inner parameters. We then denote \(\theta = \bigcup_{i=1}^{M_g}\{w_i\} \cup \theta_i\), where \(w_i \in \mathbb{R}\) for \(i=1,\dots, M_g\) are the outer parameters or outer weights used when defining the resulting kernel,
\begin{align}\label{eq:weight_sum}
    g_{\theta} =\sum_{i=1}^{M_g} \Tilde{w}_i^2 g_{\theta_i} = \sum_{i=1}^{M_g} \left(\frac{w_i}{\Bar{w}}\right)^2 g_{\theta_i}.
\end{align}
Here, \(\Bar{w} = \sum_{i=1}^{M_g} \lvert w_i\rvert \), which means we normalize the weights. As \(\Tilde w_i^2 \geq 0\), we have that \(g_{\theta}\) preserves common kernel properties such as positive definiteness. We also want to encourage sparsity in kernels and therefore we opt to regularize the outer parameters with \(L_1\)-regularization and the inner parameters with an \(L_2\)-regularization when performing dictionary learning, resulting in the following loss function: 
\begin{align*}
    \mathcal{L}^{sk}_{(\beta_1, \beta_2)}(\theta) = \mathcal{L}^{sk}(\theta) + \beta_1 \sum_{i=1}^{M_g} \lvert w_i \rvert + \beta_2 \sum_{i=1}^{M_g} \lVert \theta_i \rVert^2_2,
\end{align*}
where \(\beta_1 \in \mathbb{R}_{>0}\) and \(\beta_2 \in \mathbb{R}_{>0}\) are the regularization constants. A meaningful addition to the dictionary learning due to the regularization is that we can prune the kernel \(g_{\theta}\). That is, after a certain amount of training, evaluate the outer parameters and remove the kernel with the smallest value \(\lvert w_i \rvert\) (or remove all kernels below a certain threshold), and then continue the training with the remaining kernels. When continuing, one may either reset the parameters of the remaining kernels to their initialization value, or continue with the parameters found until that point. One can then either finish training, or again complete a certain amount of training, and then remove more kernels. Although we here apply the \(L_1\) norm when pruning, if a sparse kernel representation is the goal, one could also consider techniques closer to iteratively reweighted least squares (IRLS) to avoid the computational burden relating to the \(L_1\) norm \citep{daubechies-2010, owhadi-2019b}. We will apply pruning in \Cref{sec:exp2} and \Cref{sec:exp3}. 

\subsubsection{Batching, subsampling, and regularization}
A few additions are needed to make kEDMD-DL applicable and be computationally feasible. First, in line with most machine learning approaches, we have found that using stochastic gradient descent (SGD) rather than gradient descent improves the generalization (test) error. To use SGD, we specify a number of batches and divide the training points in \((X, Y)\) evenly across them. After performing one gradient step for each batch, known as an epoch, we divide the training points differently across the same number of batches. In our experiments so far, this has been beneficial to achieve convergence. Although not completely understood, we hypothesize that, because kernels have fewer parameters to train on and many kernels involve parameters inside an exponential function (such as the RBF kernel), this can lead to quite sharp and difficult loss landscapes. Through SGD, we effectively change the loss landscape for each batch, which can mitigate some problems, such as local minima.

Second, in the dictionary \(\Psi^{sk}_{\theta}\), defined in \Cref{eq:skedmd_dict}, we have \(N\) elements, one for each point in \(X\). The larger \(N\) is, the more computational resources are required. In addition to this, every point in every batch during training is then represented in the dictionary, and therefore, in the Koopman approximation \(K_{sk}\). To both alleviate the computational burden, and promote generalization through the training, we subsample \(\Tilde N\) points from \(X\) and \(Y\) to create the sets \(\Tilde X \subseteq X\) and \(\Tilde Y \subseteq Y\), which we then use to compute \(K_{sk}\) following \Cref{eq:k_simplified}, i.e., 
\begin{align*}
    K_{sk} = g_{\theta}(\Tilde Y, \Tilde X) g_{\theta}(\Tilde X, \Tilde X)^+,
\end{align*}
where \(K_{sk} \in \mathbb{R}^{\Tilde N \times \Tilde N}\).

Third, in addition to adding regularization to the parameters \(\theta\), we also have regularization when solving the least-squares problem to acquire the approximation \(K_{sk}\) and the projection matrix \(C_p\) or the Koopman modes \(C_m\). Hence, we have four regularization constants, which appear manageable, as we have not needed much fine-tuning in our experiments for three of them. However, the crucial regularization constant \(\beta_{koop}>0\) used when computing \(K_{sk}\) requires a bit more attention. This is, if the initialization is particularly bad, the need for regularizing when computing \(K_{sk}\) is large, and we found that implementing a simple scheduler which decreased the regularization constant \(\beta_{koop}\) after a certain number of epochs has been quite beneficial. Otherwise, if we choose a small fixed \(\beta_{koop}\), the loss function might explode, and a large fixed \(\beta_{koop}\) might lead to plateauing of the loss.

\begin{minipage}{1.0\linewidth}
    \begin{algorithm}[H]
        \caption{Dictionary learning for kEDMD, using a parameterized kernel \(g_{\theta}\). As input, it takes snapshots from the dynamical system \((X,Y)\), initial parameter value \(\theta_0\), learning rate \(\eta > 0\), regularization constants \(\beta_1, \beta_2, \beta_{modes} > 0\), in addition a \(\texttt{Scheduler}\) for the regularization constant \(\beta_{koop} >0\), and number of epochs.}\label{alg:dl_kedmd}
        \begin{algorithmic}
            \Procedure{DL-EDMD}{$X$, $Y$, $\theta_0$, $\eta$, $\beta_1$, $\beta_2$, $\beta_{modes}$, $\texttt{Scheduler}$, $N_{\text{epochs}}$}
            \State $\Tilde X, \Tilde Y \leftarrow \text{Subsample}(X,Y)$
            \State $\theta \leftarrow \theta_0$
            \For{$e=1,2,\dots, N_{\text{epochs}}$}
                \State $\beta_{koop} \rightarrow \texttt{Scheduler}(e)$
                \For{$(X_b,Y_b) \in \texttt{Batch}(X,Y)$}
                    \State $K \leftarrow g_{\theta}(\Tilde Y, \Tilde X) (g_{\theta}(\Tilde X, \Tilde X) + \beta_{koop} I)^+$
                    \State $C_p \leftarrow \Tilde X\tran (g_{\theta}(\Tilde X, \Tilde X) + \beta_{modes} I)^+$
                    \State Hold \(C_p\) and \(K\) fixed w.r.t. gradients below.
                    \State $\theta = \theta - \eta \frac{\partial}{\partial \theta} (\lVert Y_b\tran - C_pK g(X_b, \Tilde X)\rVert^2_F +  \beta_1\sum_{i=1}^{M_g}  \lvert w_i\rvert + \beta_2 \lVert \theta_i\rVert^2_2)$
            \EndFor
            \EndFor
            \State Return $\theta$
            \EndProcedure
    \end{algorithmic}
    \end{algorithm}
\end{minipage}

Finally, to stabilize the losses in the dictionary space, i.e., the dictionary loss \(\mathcal{L}_{dict}^{sk}\) and eigen loss \(\mathcal{L}_{eig}^{sk}\), we divide by the integral of the dictionary, and redefine the losses as
\begin{align*}
   \mathcal{L}_{dict}^{sk}(\theta) = \frac{\lVert \Psi^{sk}_{\theta}(Y) - K_{sk} \Psi^{sk}_{\theta}(X)\rVert_F^2}{(\int_{\mathcal{X}} \lVert \Psi_{\theta}\rVert d\mu)^2},\quad 
   \mathcal{L}_{eig}^{sk}(\theta) = \frac{\lVert \Phi^{sk}_{\theta}(Y) - \Lambda \Phi^{sk}_{\theta}(X) \rVert_F^2}{(\int_{\mathcal{X}} \lVert \Psi_{\theta}\rVert d\mu)^2}.
\end{align*}
We naturally do not have access to the integral, and rather use a certain number of points from \((\Tilde X, \Tilde Y)\) and compute an estimate based on this. 

This completes the full algorithm we propose for dictionary learning for kEDMD and is summarized in \Cref{alg:dl_kedmd}, using only \(\mathcal{L}_{pred}^{sk}\) as the loss function. Note that we hold \(K\) and \(C_p\) fixed when computing loss and gradients, as done by \citet{li-2017}. We now proceed to empirically evaluate the algorithm's performance.

\section{Experiments}\label{sec:experiments}
In this section, we will evaluate empirically the algorithm kEDMD-DL, proposed in \Cref{sec:method}. We start by considering a single RBF kernel in \Cref{sec:exp1} and compare the loss functions for the original kEDMD and the simplified kEDMD. In \Cref{sec:exp2}, we evaluate whether or not the algorithm can identify the good kernels among several kernels. In \Cref{sec:exp3} we consider the Kuramoto-Sivashinsky (KS) partial differential equation (PDE) and perform dictionary learning over a larger sum of kernels to see how the kEDMD-DL algorithm performs on a more difficult example. In \Cref{tab:hyperparameters}, we list the relevant hyperparameters and the values used for each of the three experiments we ran. The code can be found on \url{https://gitlab.com/fd-research/kerneledmd-dl}.
\begin{table}[ht]
    \caption{The hyperparameters chosen for each experiment. If there are more than one value given under the Koopman regularization, it means we have used a scheduler; exactly when the change of regularization value occurs is given in the text of the corresponding experiment.}\label{tab:hyperparameters}
    \centering
    \begin{tabular}{llll}
        \toprule
        Hyperparameters & Duffing& Modulo & KS \\
        \midrule
        Learning rate  & 5e-4 & 1e-1 & 5e-3 \\
        Number of epochs & 20 & 15 & 200 \\
        Size of \(K\) (\(\Tilde N\)) & 40 & 40 & 100 \\
        Koopman regularization & 1e-8 & 1e-6, 1e-8 & 1e-4, 1e-8 \\
        Modes regularization & 1e-8 & 1e-8 & 1e-8 \\
        $L^1$-regularization & 1e-8 & 1e-8 & 1e-8 \\
        $L^2$-regularization & 1e-8 & 1e-8 & 1e-8 \\
        \bottomrule
    \end{tabular}
\end{table}

\subsection{Dictionary learning using kernels with one parameter}\label{sec:exp1}
In the first experiment, we consider the Duffing oscillator, where the evolution of \(\vx = [x_1, x_2]\) is described through the two differential equations,
\begin{align*}
    \dot{x_1} &= x_2 \\
    \dot{x_2} = -\delta_1 x_2 - &\delta_2 x_1 - \delta_3 x_1^3.
\end{align*}
In this experiment, we set \((\delta_1, \delta_2, \delta_3) = (0.5, -1.0, 1.0)\), which gives us a system with two stable steady states. When generating data, we sample initial conditions from the cube \([-2,2]^2\). We draw 50 initial conditions for the training set and evolve each them for 10 steps from \(t=0.0\) to \(t=1.0\), leaving us with 500 training points in total, which we then divide into five batches. In addition, we create a test set with 100 initial conditions, leaving us with 1000 test points. For the size of the dictionary, we set \(\Tilde N = 40\), i.e., we sample 40 training points at the beginning of training to create \((\Tilde X, \Tilde Y)\). Following the hyperparameters given in \Cref{tab:hyperparameters}, we will train the radial basis function (RBF) kernel, with its sole parameter \(\sigma\), 
\begin{align*}
    g(\vx,\vy) = \exp\left(-\frac{\lVert \vx-\vy\rVert^2}{2\sigma^2}\right),
\end{align*}
where we initialize the parameter to \(\sigma = 1000\), making the behavior of the kernel close to that of a linear kernel. After the training, we end up with \(\sigma = 6.34\). 
\begin{figure}[ht]
    \centering
    \begin{subfigure}{0.3\textwidth}
        \includegraphics[width=\linewidth]{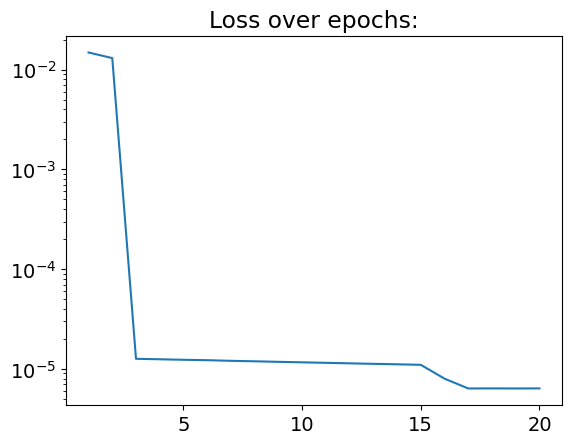}
    \end{subfigure}
    \hfill
    \begin{subfigure}{0.3\textwidth}
        \includegraphics[width=\linewidth]{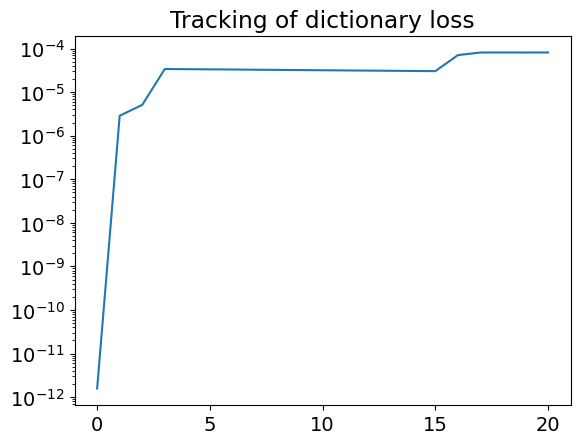}
    \end{subfigure}
    \hfill
    \begin{subfigure}{0.3\textwidth}
        \includegraphics[width=\linewidth]{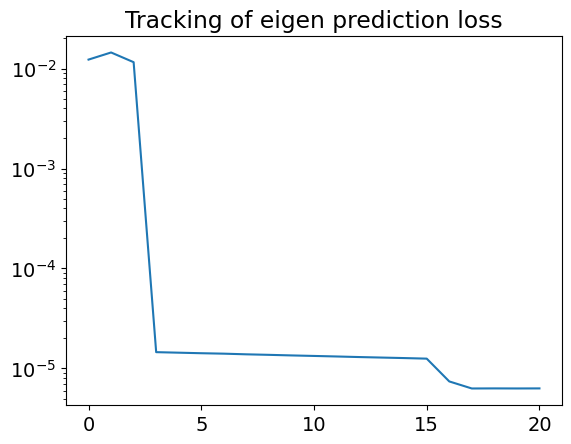}
    \end{subfigure}
    \caption{Different loss functions (vertical axes) measured over training epochs (horizontal axes), for the Duffing oscillator. On the left, we have \(\mathcal{L}^{sk}_{pred}\), which is also used as the loss function in \Cref{alg:dl_kedmd}. In the middle, we have the dictionary loss \(\mathcal{L}^{sk}_{dict}\). On the right, we have the eigen prediction loss \(\mathcal{L}^{sk}_{eig-pred}\).}
    \label{fig:duffing_loss}
\end{figure}

In \Cref{fig:duffing_loss}, we calculate the mean loss over all five batches for each epoch. We observe that the prediction loss and eigen prediction loss agree, and decrease rapidly, before somewhat stabilizing around \(10^{-5}\). On the other hand, the dictionary loss is very low from the start, with the first point evaluated before any gradient steps have been taken. If we used the dictionary loss, the parameter would remain fairly stable at the initial \(\sigma=1000\). We also observe a similar effect when tracking \(\mathcal{L}^{sk}_{eig}\) as when tracking \(\mathcal{L}^{sk}_{dict}\). This shows the need for a different loss function, such as the prediction loss. Although one should also mention that this effect of a very low starting loss for \(\mathcal{L}^{sk}_{dict}\) might appear more often when we only have one or a few parameters to optimize over. Another aspect of \Cref{fig:duffing_loss} is the behavior of little to no change, before rapidly decreasing, as we see in \(\mathcal{L}^{sk}_{pred}\). This might partially be due to optimizing over a single parameter within an exponential function. Informally, the window where the kernel has a significant effect on a certain point depends on the distance to the points in the dataset \((\Tilde X, \Tilde Y)\), and exactly how close they need to be depends on \(\sigma\). Hence, a small change in \(\sigma\) may have a very small or very large effect, which again makes the loss landscape more rugged. This effect can be somewhat mitigated by choosing different kernels or, even better, adding more kernels to a weighted sum, thereby increasing the dimensionality of the parameter space. An increased dimensionality can, in fact, have a beneficial effect on the loss landscape, as is commonly observed in machine learning \citep{cooper-2021, hastie-2022}.

We then move on to consider the trajectories computed from the test set's initial conditions. In \Cref{fig:duffing_trajs}, we see the trajectories computed using both the kernel with initialization parameter \(\sigma=1000\) and the kernel learned via dictionary learning. The parameter found by kEDMD-DL captures the underlying dynamics well, even when training starts with a kernel that clearly does not.
\begin{figure}[ht]
    \centering
    \includegraphics[width=\linewidth]{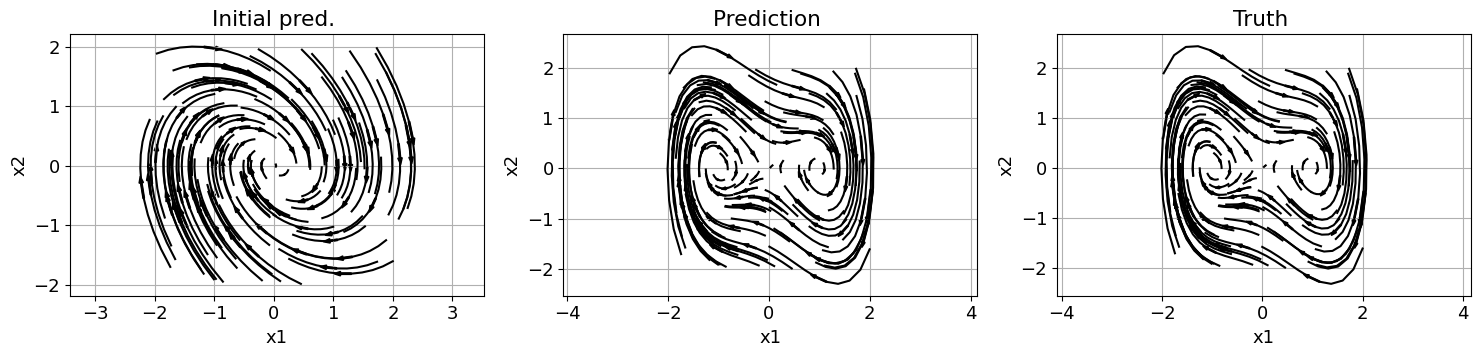}
    \caption{The trajectories of the test set, where on the left, we see the trajectories computed from the initial conditions using the kernel with the initial parameter \(\sigma=1000\), in the middle is the prediction using the kernel found through dictionary learning, and on the right, we see true trajectories.}
    \label{fig:duffing_trajs}
\end{figure}
\begin{figure}[ht]
    \centering
    \begin{subfigure}{0.35\textwidth}
        \includegraphics[width=\linewidth]{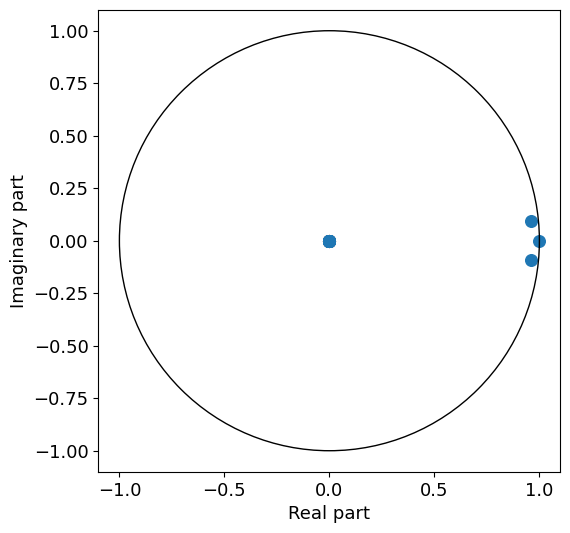}
    \end{subfigure}
    \hspace{1cm}
    \begin{subfigure}{0.35\textwidth}
        \includegraphics[width=\linewidth]{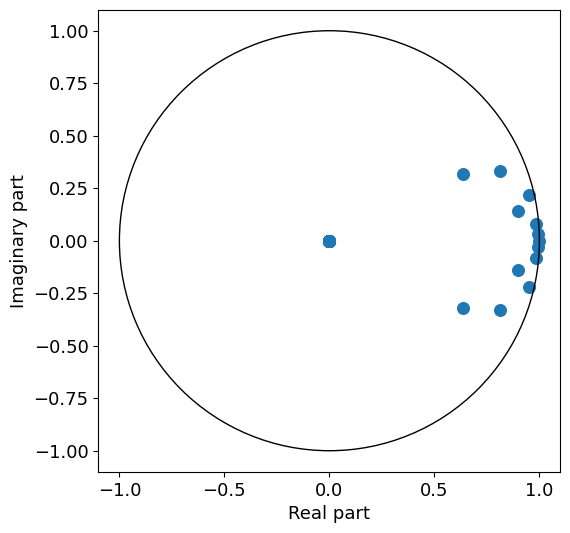}
    \end{subfigure}
    \caption{The eigenvalues of the Koopman approximation for the Duffing oscillator, using two different kernels. The plot on the left shows the eigenvalues before training, i.e., the RBF kernel with \(\sigma=1000\). The plot on the right shows the eigenvalues found after training. The x- and y-axis correspond to the real and imaginary part of the eigenvalue, respectively.}
    \label{fig:duffing_spect}
\end{figure}
If we then consider the eigenvalues we learn and compare it with the initial kernel, which we have plotted in \Cref{fig:duffing_spect}, we see that we have many more distinct nonzero eigenvalues after performing dictionary learning as opposed to beforehand. This also gives us an insight into why using dictionary-based losses, such as dictionary loss and eigen loss, alone is not always the right choice. Finally, the two different plots also show that we start and end with quite well-behaved kernels, that is, kernels that do not lead to eigenvalues far outside the unit circle.

\subsubsection{Comparing with the original kEDMD}
When we trained the kernel, in addition to tracking all the loss functions above, we also computed the loss functions based on the original kEDMD. That is, for every batch, we computed \(K_{tr}\) using the same kernel and parameter value \(\sigma\) as when we tracked the different loss functions in \Cref{fig:duffing_loss}. The resulting plots can be found in \Cref{fig:duffing_loss_trunc} and show similar behavior for the eigen prediction loss as with the simplified kEDMD version. The eigen loss also shows the unwanted behavior that we sometimes observe with these dictionary-based loss functions. In contrast, the dictionary loss derived in \Cref{sec:loss_trunc_kedmd} shows a different, slightly more beneficial behavior (beneficial as in more suitable to be used in a dictionary learning scheme). It is hard to pinpoint exactly why this happened here, but perhaps part of the behavior stems from failing to satisfy the full-rank assumption used when deriving the dictionary loss function for the original kEDMD. In addition, the regularization used in numerical experiments may induce some differences from the unregularized versions used when proving \Cref{prop:simp_kedmd}.  In general, if one were to apply dictionary learning to original kEDMD, an eigen prediction loss \(\mathcal{L}^{tr}_{eig-pred}\) seems to be the way to go, as we do not have access to a prediction loss.

\begin{figure}[ht]
    \centering
    \begin{subfigure}{0.3\textwidth}
        \includegraphics[width=\linewidth]{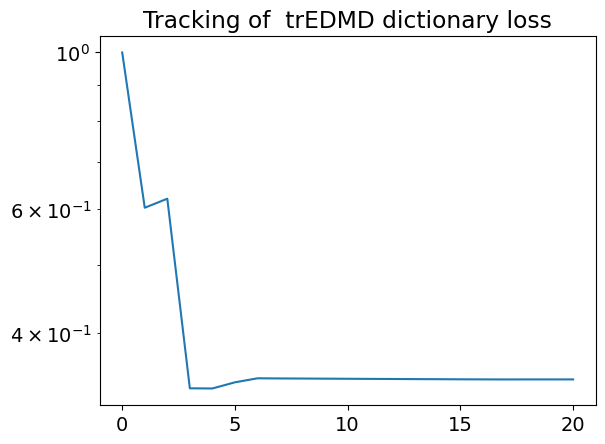}
    \end{subfigure}
    \hfill
    \begin{subfigure}{0.3\textwidth}
        \includegraphics[width=\linewidth]{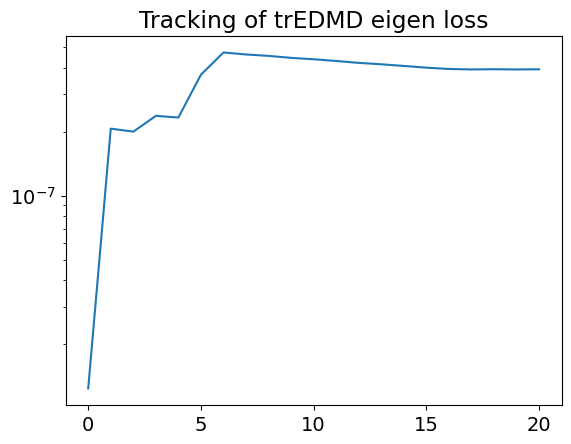}
    \end{subfigure}
    \hfill
    \begin{subfigure}{0.3\textwidth}
        \includegraphics[width=\linewidth]{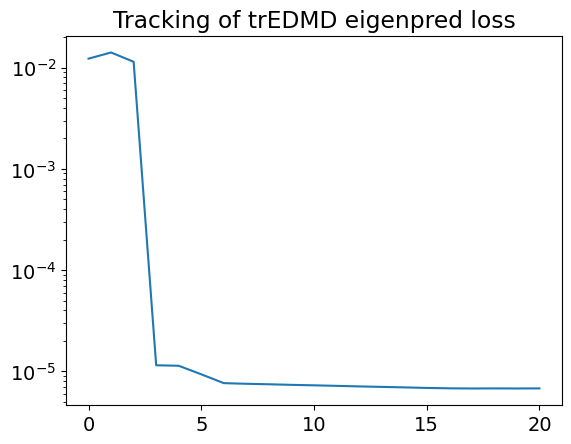}
    \end{subfigure}
    \caption{Different loss functions (vertical axes) measured for each epoch (horizontal axes), using the current kernel found by \Cref{alg:dl_kedmd} to compute \(K_{tr}\). On the left, we show the dictionary loss \(\mathcal{L}^{tr}_{dict}\), in the middle, the eigen loss \(\mathcal{L}^{tr}_{eig}\), and on the right, the eigen prediction loss \(\mathcal{L}^{tr}_{eig-pred}\).}
    \label{fig:duffing_loss_trunc}
\end{figure}

If we then, as a sanity check, consider the trajectories in \Cref{fig:duffing_trajs_trunc}, which we achieve by passing the kernel with the parameter at initialization and the parameter after training to the original kEDMD to compute \(K_{tr}\), we observe similar behavior as in \Cref{fig:duffing_trajs}. This is also seen in the case of the eigenvalues, plotted in \Cref{fig:duffing_spect_trunc}. 
\begin{figure}[ht]
    \centering
    \includegraphics[width=\linewidth]{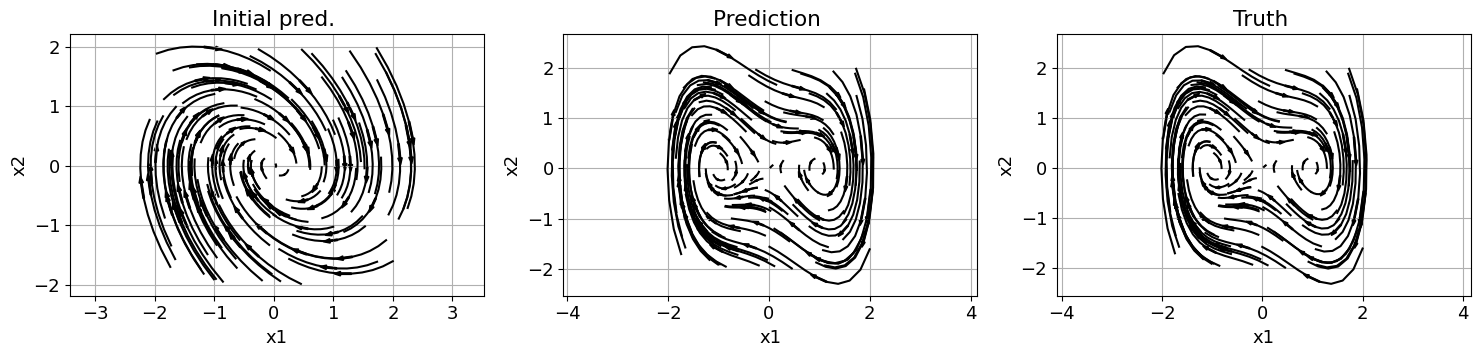}
    \caption{The trajectories of the test set, where on the right, we see the true trajectories. Using the original kEDMD, on the left, we see the trajectories computed from the initial conditions using the kernel with the initial parameter \(\sigma=1000\), and in the middle is the prediction using the kernel found through dictionary learning.}
    \label{fig:duffing_trajs_trunc}
\end{figure}
\begin{figure}[ht]
    \centering
    \begin{subfigure}{0.35\textwidth}
        \includegraphics[width=\linewidth]{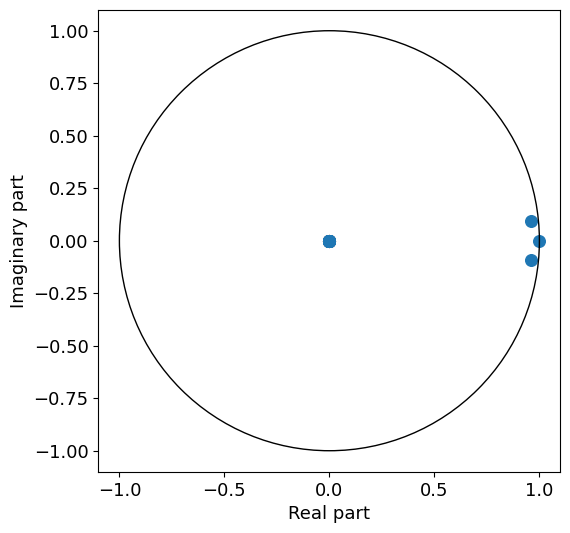}
    \end{subfigure}
    \hspace{1cm}
    \begin{subfigure}{0.35\textwidth}
        \includegraphics[width=\linewidth]{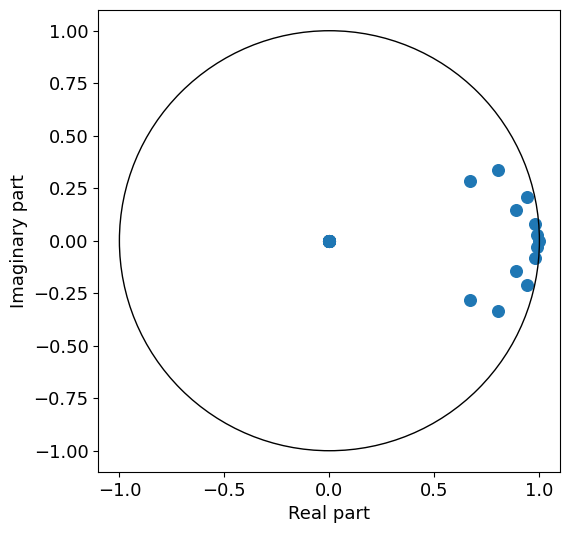}
    \end{subfigure}
    \caption{The eigenvalues of \(K_{tr}\) computed using the kernel with initialization parameter on the left, and similarly using the parameter found through dictionary learning on the right.}
    \label{fig:duffing_spect_trunc}
\end{figure}

\subsection{Identifying the correct kernel}\label{sec:exp2}
We will now start adding kernels to create a weighted sum as in \Cref{eq:weight_sum}. One goal of kEDMD-DL is to identify the useful kernels in a sum. We therefore want the training to not only find suitable inner parameters, but also distinguish the kernels through the outer parameters. In a typical setting, it is not easy to identify pre-training what a good kernel looks like. In this experiment, we construct a dynamical system and a list of kernels, where some kernels are clearly more useful than others. We start by considering a system described by
\begin{align}\label{eq:modulo_sys}
    x_{t+1} = (x_{t} + \omega) \text{ mod } 2\pi,
\end{align}
where \(\omega \in \mathbb{R}\). This system can be seen as an angle of a circle, where at each time \(t\), \(\omega\) is added to the previous angle \(x_{t}\). While mapping it to a point on a circle leaves us with a dynamical system that is easy to capture, working in the Euclidean space \(\mathbb{R}\) is more challenging, as the state \(x_{t}\) will suddenly jump. For our specific setting, we set the \(\omega = 1.1\pi\). We draw 100 initial conditions uniformly from \([0, 2\pi)\), evolving them 50 time steps ahead, leaving us with 5000 data points in the training set, where we divide them into five batches. For the subsampling, we set \(\Tilde N = 40\), and create the following kernel:
\begin{align*}
    g_{\theta}(x,y) = \Tilde w_{1}^2 \exp\left(-\frac{\lVert h(x)-h(y)\rVert^2_2}{2\sigma_{1}^2}\right) + \Tilde w_{2}^2 \exp\left(-\frac{\lvert x-y\rvert^2}{2\sigma_{2}^2}\right) + \Tilde w_3^2 \cos(a \lvert x-y \rvert^2) + \Tilde w_4^2 (c_1^2 xy),
\end{align*}
where \(h\) creates a two-dimensional vector from the input point, i.e., \(h(x) = [\cos(x), \sin(x)]\). Although the linear kernel, the RBF kernel, and the cosine kernel may be somewhat useful, their performance  is negatively affected by the discontinuous jump when \((x_t + \omega) > 2\pi\), while the embedding created by \(h\) erases the discontinuity. So after the training, the key part we look for is the outer parameters and the relationship between the three kernels. We initialize the outer parameters to \(w_1=w_2=w_3=w_4=0.25\), and after performing dictionary learning, we end up with \((w_1, w_2, w_3, w_4) = (-5.482, -1.072, -0.029, 0.710)\); hence, the kernel with the embedding is strongly emphasized. Looking at the eigenvalues in \Cref{fig:modulo_spectrum}, we find that the initial kernel is particularly inaccurate, with several eigenvalues well outside of the unit circle. This is also why we use a scheduler for the Koopman regularization, changing its value after five steps. After training, we obtain an improved collection of eigenvalues. In particular, for the system described by \Cref{eq:modulo_sys}, the eigenvalues we want to approximate are of the form \(\exp(i\omega j)\), \(j\in\mathbb{Z}\). If we take it a step further, based on the values of the outer parameters, and prune the kernel \(g_{\theta}\) by setting \(w_2=w_3=w_4=0.0\), we achieve the plot on the right, where the eigenvalues are either zero or on the unit circle. We also plot the first 20 true eigenvalues of the system, that is, we plot \(\exp(i\omega j)\) for \(j=1,2,\dots,20\). Notice that the seven nonzero eigenvalues are each exactly the same as a true value (both the real and imaginary component are within 1e-3 of a true eigenvalue). More specifically, the nonzero eigenvalues of the approximation are exactly \(\exp(i\omega j)\), for \(j=0,1,2,3,17,18,19\). To find more nonzero eigenvalues, we would need to increase \(\Tilde N\). Still, this shows that the algorithm presented seems to find valuable kernels, and we will use it for a more challenging problem next.

\begin{figure}[ht]
    \centering
    \begin{subfigure}{0.27\textwidth}
        \includegraphics[width=\linewidth]{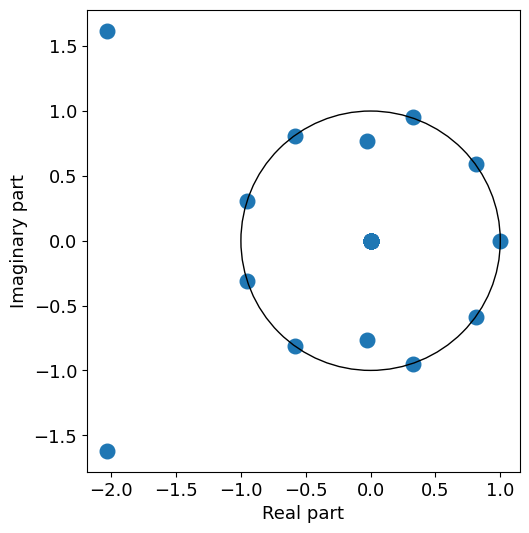}
    \end{subfigure}
    \hfill
    \begin{subfigure}{0.3\textwidth}
        \includegraphics[width=\linewidth]{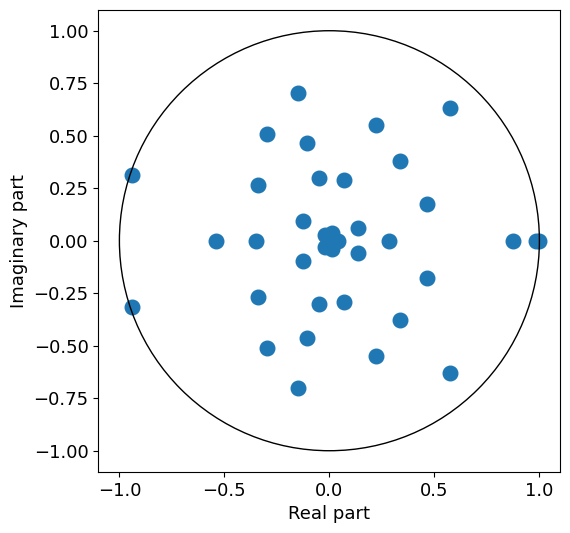}
    \end{subfigure}
    \hfill
    \begin{subfigure}{0.3\textwidth}
        \includegraphics[width=\linewidth]{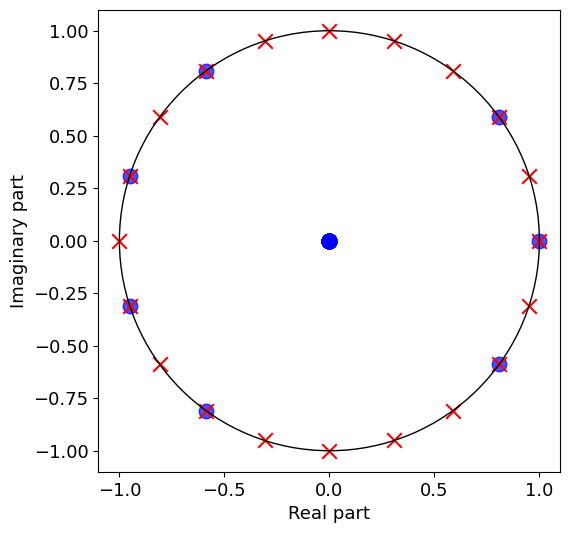}
    \end{subfigure}
    \caption{Eigenvalues of \(K_{sk}\), where on the left, we have before training, in the middle after training, and on the right is after pruning, i.e., after setting \(w_{2}=w_{3}=w_4=0.0\). On the right plot, we also plot the first 20 true eigenvalues, i.e., we plot \(\exp(i\omega j)\) for \(j=1,2,\dots,20\). The true values are marked by red cross, while the approximated ones are the blue dots.}
    \label{fig:modulo_spectrum}
\end{figure}

\subsection{Kuramoto-Sivashinsky equation}\label{sec:exp3}
To test the kEDMD-DL algorithm for a larger sum of kernels, we opt to work with a system defined by the Kuramoto-Sivashinsky equation (KSE), namely,
\begin{align*}
    u_t + 4u_{xxxx} + \gamma (u_{xx} + uu_x) = 0, 
\end{align*}
where \(x \in [0,1]\). We follow the setup of \citet{li-2017} and set \(\gamma = 16\) and with periodic boundary condition. For the initial condition, we let
\begin{align*}
    u(x,0) = \tau_1 \sin(2\pi x) + \tau_2 \exp(\cos(2\pi x)),
\end{align*}
where \(\tau_1 \in [0.8, 1]\) and \(\tau_2 \in [0.5,1]\). We generate training data by uniformly sampling 100 pairs of points \((\tau_1, \tau_2)\), and from there generate 100 initial conditions evaluated at 50 equally spaced points \(0 \leq x_1 < \dots < x_{50} \leq 2\pi\). We then compute 100 timesteps ahead, from \(t=0\) to \(t=0.5\). We similarly sample ten pairs of points \((\tau_1, \tau_2)\) and generate our test data in the same fashion. We use 4 batches.

For our kernel selection, we choose to add a few RBF kernels, one cosine kernel, one linear kernel, and a neural network Gaussian process (NNGP) kernel. For the latter, we briefly introduce the latter, as the reader might not be as familiar with this kernel. If you have a neural network with Gaussian initialization and take the limit as the number of hidden-layer neurons approaches infinity, the output neurons follow a Gaussian process (GP). For a certain activation function, one even has a closed-form solution for the kernel that defines the GP, which is then the NNGP kernel. For more on NNGP, we refer the reader to \citet{rasmussen-2005}, \citet{avidan-2024}, and \citet{lee-2018}. We base our kernel on the closed-form solution of an NNGP kernel given by \citet{lee-2018}, which yields the NNGP kernel with ReLU activation function. We start by letting
\begin{align*}
    g^{0}_{(b_1, b_2)}(\vx,\vy) = b_1^2 \frac{\langle \vx, \vy\rangle}{d} + b_2^2,
\end{align*}
where \(d\) is the dimension of the input, which in our case is \(50\). We then let
\begin{align*}
    g^{NNGP}_{(b_1, b_2)}(\vx, \vy) &= b_2^2 + \frac{b_1^2}{2\pi} \sqrt{g^0_{(b_1,b_2)}(\vx) g^0_{(b_1,b_2)}(\vy)}
    \left(\sin \omega_{\vx, \vy} + (\pi - \omega_{\vx,\vy})\cos\omega_{\vx,\vy}\right)\\
    \omega_{\vx,\vy} &= \cos^{-1}\left(\frac{g^0_{(b_1, b_2)}(\vx,\vy)}{\sqrt{g^0_{(b_1, b_2)}(\vx) g^0_{(b_1, b_2)}(\vy)}}\right).
\end{align*}
Including this kernel not only allows for a possibly better fit but also nicely connects kEDMD-DL to EDMD-DL by \citet{li-2017}, as in EDMD-DL, one often uses neurons as dictionary elements. 

\begin{figure}[ht]
    \centering
    \begin{subfigure}{0.3\textwidth}
        \includegraphics[width=\linewidth]{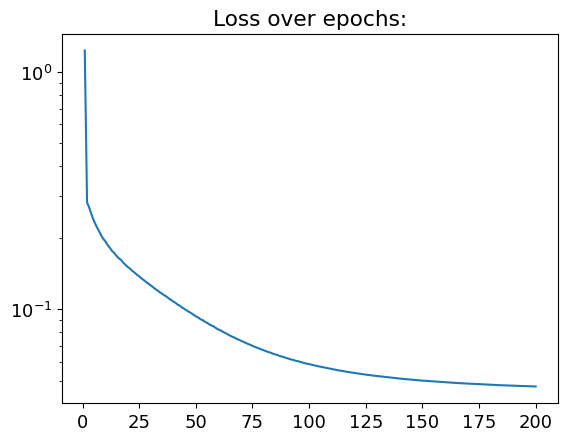}
    \end{subfigure}
    \hfill
    \begin{subfigure}{0.3\textwidth}
        \includegraphics[width=\linewidth]{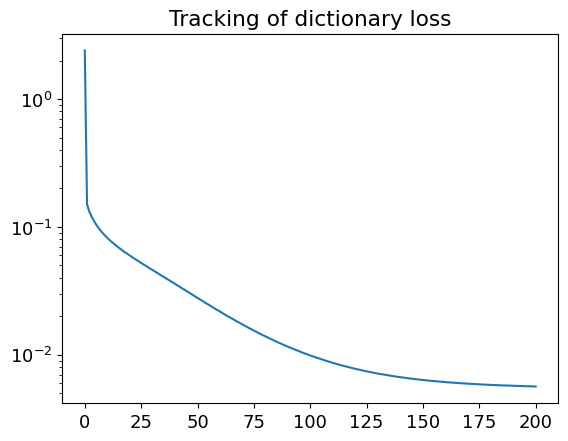}
    \end{subfigure}
    \hfill
    \begin{subfigure}{0.3\textwidth}
        \includegraphics[width=\linewidth]{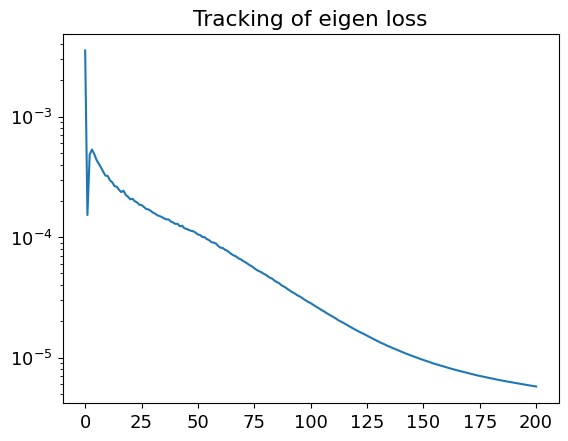}
    \end{subfigure}
    \caption{Different loss functions measured for each epoch for the KS PDE. On the left, we have \(\mathcal{L}^{sk}_{pred}\). In the middle, we have the dictionary loss \(\mathcal{L}^{sk}_{dict}\). On the right, we have the eigen loss \(\mathcal{L}^{sk}_{eig}\).}
    \label{fig:kse_loss}
\end{figure}

\begin{figure}[ht]
    \centering
    \begin{subfigure}{0.90\textwidth}
        \includegraphics[width=\linewidth]{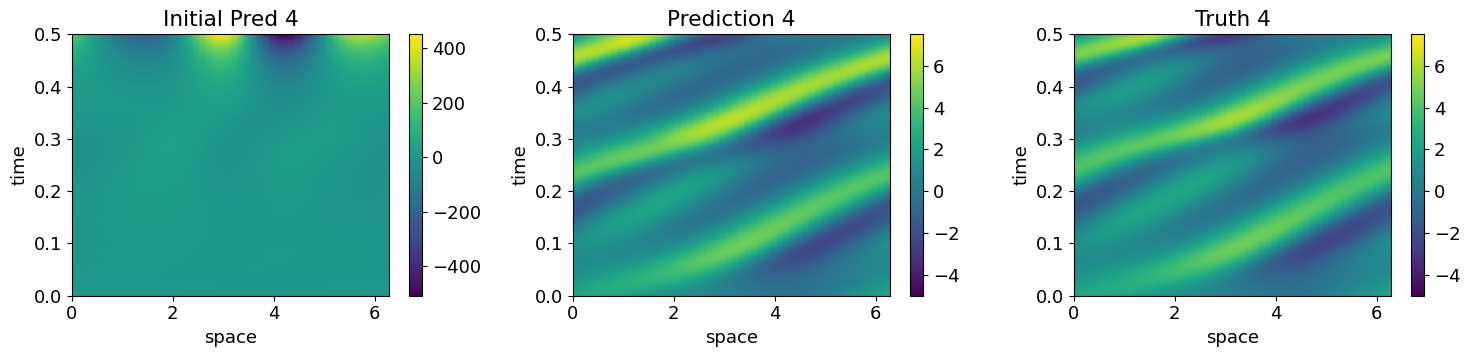}
    \end{subfigure}
    \begin{subfigure}{0.90\textwidth}
        \includegraphics[width=\linewidth]{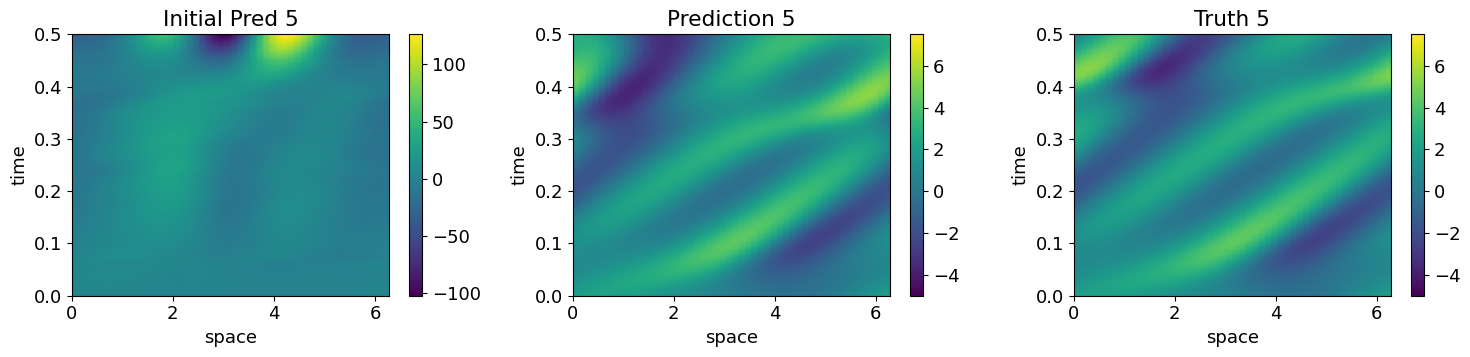}
    \end{subfigure}
    \begin{subfigure}{0.90\textwidth}
        \includegraphics[width=\linewidth]{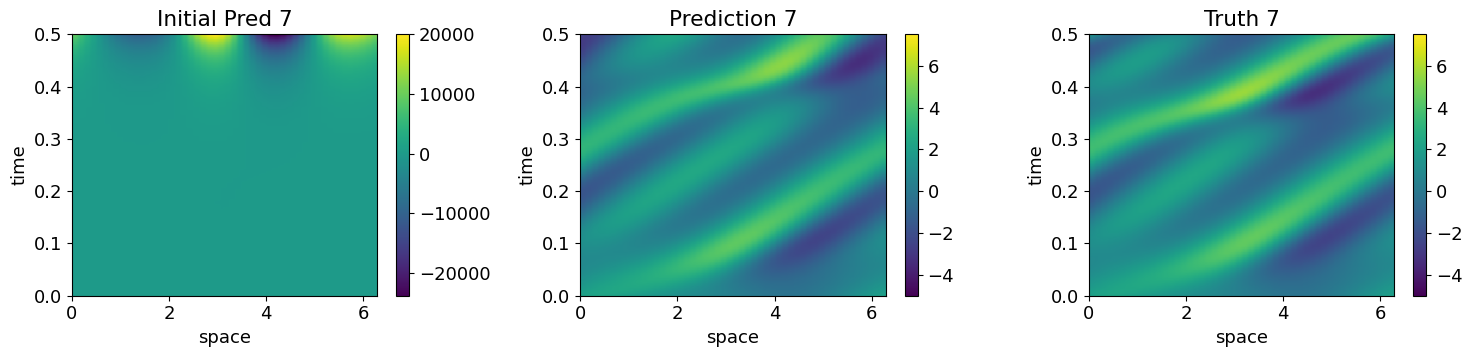}
    \end{subfigure}
    \caption{Approximating the trajectories of the Kuramoto-Sivashinsky PDE for three chosen initial conditions in the test set. On the left is the approximated trajectories using the kernel with initialization parameters, in the middle is the approximated trajectories from skEDMD after dictionary learning, and on the right is the true trajectories.}
    \label{fig:kse_test}
\end{figure}

\begin{figure}[ht]
    \centering
    \begin{subfigure}{0.55\textwidth}
        \includegraphics[width=\linewidth]{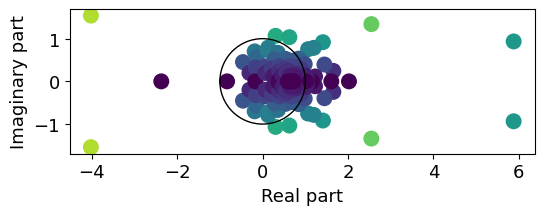}
    \end{subfigure}
    \hspace{1cm}
    \begin{subfigure}{0.35\textwidth}
        \includegraphics[width=\linewidth]{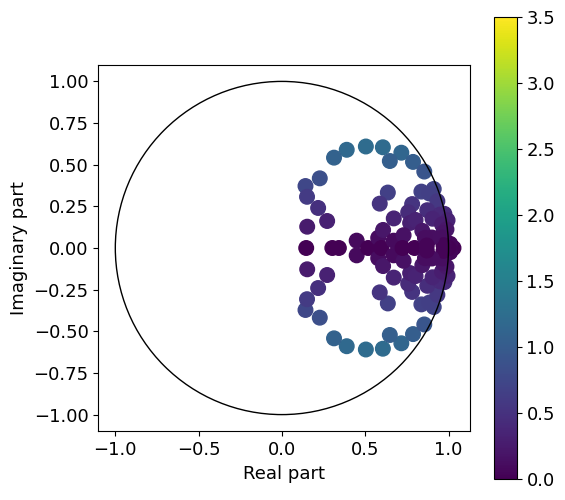}
    \end{subfigure}
    \caption{The eigenvalues of \(K_{sk}\) when approximating the Kuramoto-Sivashinsky PDE, computed using kernel with initialization parameter on the left, and similarly for the parameter found through dictionary learning on the right. The corresponding color of each point is the eigenvalue's residual, computed following \Cref{sec:res_spec}. The minimum, median, and maximum value of the residuals for the initial kernel and after training are \((\text{2.03e-7}, 0.60, 3.08)\) and \((\text{2.24e-7}, 0.34, 1.22)\), respectively.}
    \label{fig:kse_spectrum}
\end{figure}

Having defined our \(g^{NNGP}_{(b_1,b_2)}\), we set our kernel to be
\begin{align*}
    g_{\theta}(\vx,\vy) =  \sum_{i=1}^{3} \Tilde w_{i}^2 \exp\left(-\frac{\lVert \vx-\vy \rVert^2_2}{2\sigma_i^2}\right) + \Tilde w_4^2 \cos(a \lVert \vx-\vy \rVert^2_2) + \Tilde w_5^2 (c_1^2 \langle \vx, \vy\rangle) + \Tilde w_6^2 g^{NNGP}_{(b_1, b_2)}(\vx, \vy), 
\end{align*}
where we initialize \((\sigma_{1}, \sigma_2, \sigma_3) = (1,10, 50)\). The reason for having several RBF kernels with different parameter values is that changing the inner parameter by a significant amount inside an exponential can be challenging, as we already discussed in \Cref{sec:exp1} and also observed for general kernel learning in \citet{lengyel-2025}. So, we use a few kernels with different inner parameters and adjust them slightly, while leaving the more significant changes to the outer weights. We also initialize \((b_1, b_2, c_1, c_2, a) = (1, 0, 1, 0, 1)\), and the outer parameters \(w_i = 0.25\) for all \(i=1,2,\dots,6\). We set \(\Tilde N = 100\), slightly larger than in the previous two experiments, to accommodate the harder problem at hand, and then train for 200 epochs. The loss after training is shown in \Cref{fig:kse_loss}. The crucial part of making the training work is the scheduler. In particular, we start by regularizing \(K_{sk}\) significantly by setting the regularization constant to \(1e-4\). After a few epochs, we can already see improvement in the loss, and can then set the constant to \(1e-8\) for the remaining epochs. Without such a scheduler, the loss either blows up with a too small regularization constant, or it stops decreasing after a few epochs. It will become apparent soon why the loss will not behave nicely with a small regularization constant from the start. Another way to avoid this is to carefully select kernels and initialization values, which goes against the motivation behind dictionary learning of quickly initializing kernels without too much thought and letting the data choose the kernels and parameters for you. We also notice from \Cref{fig:kse_loss} that the dictionary-based losses are more aligned with the prediction-based loss. Compared to \Cref{sec:exp1}, this is quite different. We believe one reason for this is that we have more kernels and parameters to optimize over, and optimizing over outer parameters is significantly easier than some of the inner parameters.

After the training, we may consider first the outer parameters. We end up with \(w_1=0.267, w_2=1.149, w_3=0.356, w_4=-0.0273, w_5=0.210, w_6=1.691\), and from this we observe two kernels that the training emphasize more, namely, the RBF kernel with initial value \(\sigma_2 = 10\) (which has changed to \(\sigma_2=5.96\)) and the NNGP kernel with new inner parameters \(b_1=2.57, b_2=0\).

Looking at three of the ten trajectories in the test set in \Cref{fig:kse_test}, we see the dictionary learning has moved from an unstable kernel toward a kernel that captures the trajectories satisfactorily, although there is still some inaccuracies when the time goes toward \(0.5\). For instance, in the corner of the middle example, there are still some inaccuracies. This instability of the initialization is also highlighted by the eigenvalues in \Cref{fig:kse_spectrum}, with the learned kernel producing eigenvalues mostly inside the unit circle. For these eigenvalues, we also opted to differentiate the different eigenvalues by their residual, computed following \Cref{sec:res_spec} as we are using skEDMD, and can apply algorithms designed for EDMD directly. After training, the number of eigenvalues with high residuals is lower, and therefore the eigenpairs produced after training are a more accurate representation of the operator's true eigenvalues. The eigenvalues after training also have lower overall residuals. There are still some eigenvalues that are clearly higher in residual value than others, particularly those where the magnitude of the imaginary part is larger. However, removing eigenpairs at any threshold after training has not yielded more accurate predictions; hence, we have computed trajectories using the full eigenspace.

\begin{figure}[ht]
    \centering
    \begin{subfigure}{0.90\textwidth}
        \includegraphics[width=\linewidth]{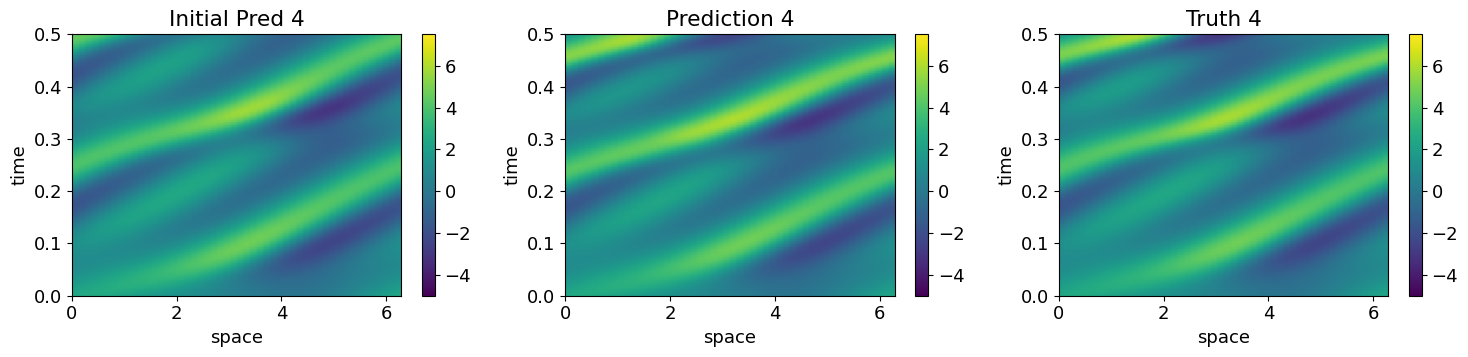}
    \end{subfigure}
    \begin{subfigure}{0.90\textwidth}
        \includegraphics[width=\linewidth]{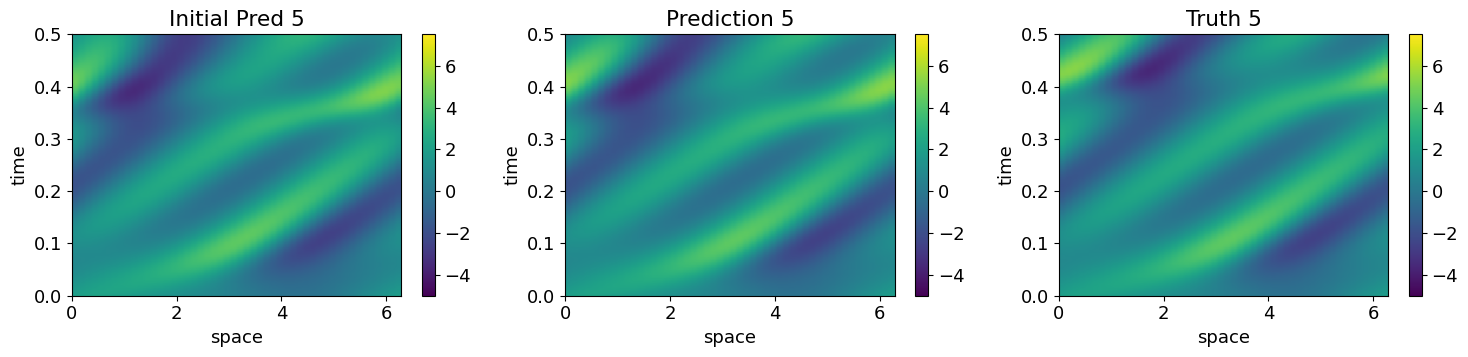}
    \end{subfigure}
    \begin{subfigure}{0.90\textwidth}
        \includegraphics[width=\linewidth]{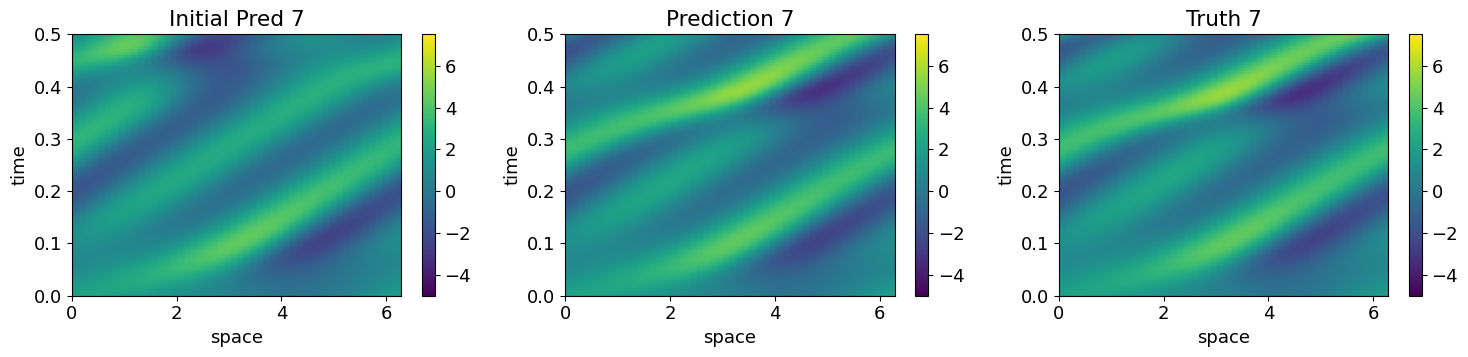}
    \end{subfigure}
    \caption{Approximating the trajectories of the Kuramoto-Sivashinsky PDE for three chosen initial conditions in the test set, after pruning the kernel. On the left is the approximated trajectory using the kernel with initialization parameters, in the middle is the approximated trajectory from kEDMD after dictionary learning, and on the right is the true trajectory.}
    \label{fig:kse_test_pruned}
\end{figure}

\begin{figure}[!ht]
    \centering
    \begin{subfigure}{0.6\textwidth}
        \includegraphics[width=\linewidth]{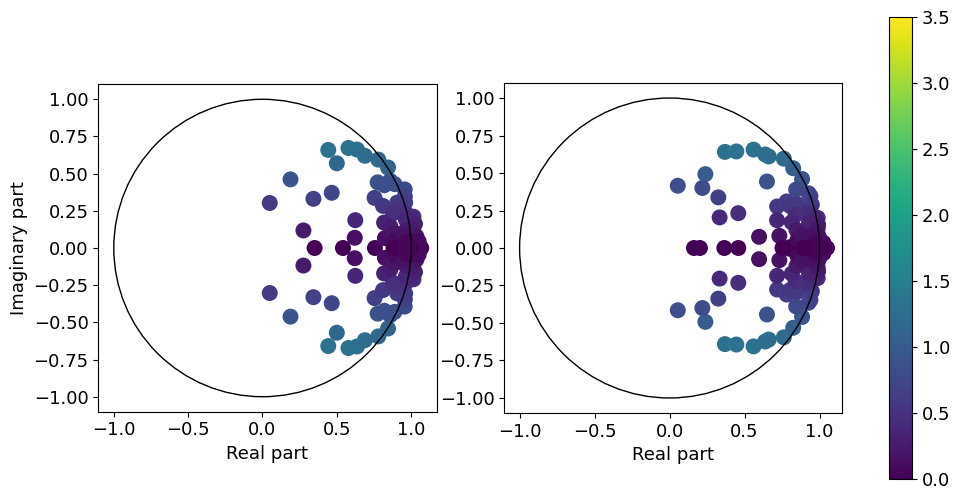}
    \end{subfigure}
    \caption{The eigenvalues of \(K_{sk}\) when approximating the Kuramoto-Sivashinsky PDE after pruning the kernel based on the first iteration of training. Computed using kernel with initialization parameter on the left, and similarly for the parameter found through dictionary learning on the right. The corresponding color of each point is the eigenvalue's residual, computed following \Cref{sec:res_spec}. The minimum, median, and maximum value of the residuals for the initial pruned kernel and after training are \((\text{1.02e-6}, 0.42, 1.35)\) and \((\text{6.87e-7}, 0.41, 1.31)\), respectively}
    \label{fig:kse_spectrum_pruned}
\end{figure}

We then proceed to prune the kernel, as in \Cref{sec:exp2}. Looking at the outer parameters, we are left with the RBF kernel associated with \(w_2\) and the NNGP kernel. After removing the other kernels used in \(g_{\theta}\), we restart training from the same initialized internal parameters of the two remaining kernels, but run the training only for 30 epochs. After training the relevant parameters are now \(w_2=0.362\), \(w_6=0.738\), \(\sigma_2=9.3\), and \((b_1, b_2) = (1.4, 0)\). The emphasis of the outer parameters is also more toward the NNGP kernel, however, not by a substantial amount. Some kernels capture part of the same effect of the system, such as the linear kernel and the NNGP kernel, which share some similarities; therefore, one would expect a somewhat different relationship between the RBF kernel and the NNGP kernel after pruning. 

In \Cref{fig:kse_test_pruned}, we plot the same three test trajectories and observe a similar trend to what we observed after the training of the original six kernels, with some inaccuracies in the middle example improved. Also, while the starting kernel is much better after pruning, there appear to be more inaccuracies in the test trajectories for the pruned kernel with initial parameter values than the trained one before pruning. Ultimately, both pruning and training have a positive effect on the approximated trajectories. We also see how the training affects the eigenvalues in \Cref{fig:kse_spectrum_pruned}, where more of the eigenvalues have moved toward the unit circle, while a few eigenvalues remain outside the unit circle, which makes sense, given that the underlying system is chaotic. The residual values are quite similar, but the trained one has lower minimum, median, and maximum values, as shown in the figure description. Finally, while the Koopman approximation with the trained kernel captures trajectories well in general, there is some additional nuance. Roughly four or five out of the 100 training trajectories diverge when approaching \(t=0.5\). By chance, none were in the test set, so we included a selection of ten training trajectories, of which two exhibit this diverging effect. The plot can be found in \Cref{appendix:test_traj}. If we then increase \(\Tilde N\) to 500 and retrain the pruned kernel with the bigger set \((\Tilde X, \Tilde Y)\) for a few epochs, we find that the issue disappears, that is, we appear to capture the dynamics more accurately. We have included the same training examples for the case with \(\Tilde N = 500\) in \Cref{appendix:test_traj}. We observe that if one needs a larger \(\Tilde N\), then training first on a smaller one for a long time, pruning the kernel, and then fine-tuning the kernel using the larger sets is a useful strategy to avoid costly computation. Also, one may consider how to sample \((\Tilde X, \Tilde Y)\) from the full dataset, and subsampling strategies may perhaps be useful to keep \(\Tilde N\) low while still capturing the dynamics.

We include some numerical results to highlight what the different parts of the training process improves upon. We repeat experiments with three different random seeds and report the median, minimum, and maximum of the mean absolute value error between the ten ground truth test trajectories and the predicted ones. We compute the error for both the kernel pre-training, after training, and pruned kernels (separately for pre- and post-final training). The results can be found in \Cref{tab:numerical_results_kse}. We find that the first training is most important, pruning improves the errors slightly, while training further the pruned kernel has an even smaller effect (the maximum value on both stems from the same run).
\begin{table}[ht]
    \caption{The mean absolute value error of the ten predicted test trajectories of the Kuramoto-Sivashinsky PDE, compared to the ground truth. The different kernels used are the ones resulting from initialization, after training, after pruning (before training), and after pruning (after training). We ran with three different random seeds, and report the median, minimum, and maximum error. The initial training is the key part, the pruning improves the errors, and subsequent training of the pruned kernel mostly improves the maximum error.}\label{tab:numerical_results_kse}
    \centering
    \begin{tabular}{llll}
        Kernel& Median error & Min. error & Max. error \\
        \hline
        Initial kern. & $8142.195$ & $5.600$ & $8.2e33$ \\
        Trained kern. & $0.471$ & $0.448$ & $0.516$ \\
        Pruned initial kern. & $0.187$ & $0.181$ & $0.454$ \\
        Pruned trained kern. & $0.185$ & $0.184$ & $0.253$ \\
    \end{tabular}
\end{table}

Before we conclude, we include a remark on the regularization scheduler. We implemented a manual version of such a scheduler, meaning we set the starting and ending regularization manually, and fixed the change after five epochs. This can be automated, and changing the regularization rate can be done by considering the previous losses and whether they keep improving significantly or not. From our experience, running for a few epochs works well, and by setting the number of epochs too high only requires a small amount of computational power. Setting it too low will immediately be clear from the exploding loss. Setting the different regularization rates was not difficult for the experiments, and simply setting the first regularization rate a few magnitudes higher than the final regularization rate worked without any hyperparameter tuning.

From this experiment, we find that starting with a larger set of kernels, performing one training iteration, pruning, resetting the parameters of the remaining kernels, and then training again yields the best results. In addition, as long as one starts with a reasonable scheduler, having an unstable kernel at the start, in which we mean a kernel leading to eigenvalues of \(K_{sk}\) being far outside of the unit circle, does not present a problem for the algorithm.

\section{Conclusion and future work}\label{sec:conclusion}
In this paper, we streamline the search for good kernels for kEDMD via stochastic gradient-based optimization. To achieve this, we simplify the kEDMD method by applying EDMD to a dictionary of kernels that meet certain requirements, rather than via truncated SVD. We show that these two methods are equivalent, while the simplified version's connection to EDMD means that tools designed for EDMD can be applied directly. We can then connect the kernel learning to dictionary learning for EDMD, with additional extensions such as subsampling, batching, and a regularization scheduler. This results in an algorithm that does not require careful initialization of the kernel parameters, as the method works well for initial approximations having eigenvalues far outside of the unit circle. Through subsampling, the size of the Koopman approximation \(\Tilde N\) is quite low in all three experiments, which is important because it can often become the computational bottleneck. We observe that the method approximates the system dynamics while pruning unimportant kernels to find accurate and efficient kernel representations. We demonstrate that the method also performs well on more challenging problems.

Current limitations of the method we have observed so far are training on less smooth kernels, such as those from the Matérn family. The need to increase \(\Tilde N\) to capture the full dynamics in \Cref{sec:exp3} may be of concern. However, including more kernels and better subsampling strategies can perhaps alleviate this, along with first training on a smaller \(\Tilde N\) and then fine-tuning on the larger \(\Tilde N\). Finally, the number of hyperparameters is quite high, which means one might need substantial computational time to fine-tune them. Fortunately, the hyperparameters have been quite easy to specify without tuning, and defining the scheduler has been straightforward by looking at the training loss without one. 

In the future, we suggest investigating the roles of the other loss functions proposed in this paper and how they can be combined to achieve even better results. We also want to consider incorporating resDMD into the training itself by computing the eigen-based losses using eigenpairs with less spectral pollution. \citet{conradie-2026} recently discussed loss terms to compute the expected \(L_2\) error, and applying this to dictionary learning for kernel EDMD seems promising. Finally, the list of kernels we have used in this paper is also quite generic, and extending it to include more kernels and to incorporate non-parametric kernel flows is of interest.

\begin{ack}
E.B. and F.D. are thankful to the Deutsche Forschungsgemeinschaft (DFG, German Research Foundation) - for funding through project no. 468830823, and also acknowledge the association to DFG-SPP-229. H.O. acknowledges support from the DoD Vannevar Bush Faculty Fellowship Program under ONR award number N000142512035. B.H. acknowledges support from the  US Department of Energy. The work of I.G.K. was supported by the US Department of Energy.
The authors are grateful to Professor Qianxiao Li and Professor Matthew Colbrook for helpful discussions, and to Iryna Burak for several comments on the manuscript.
\end{ack}

\bibliographystyle{plainnat}
\bibliography{biblio}
\appendix
\section{Appendix}
\subsection{Kuramoto-Sivashinsky test and train trajectories}\label{appendix:test_traj}
In \Cref{fig:kse_test_full}, we have added all ten trajectories from the test set, as described in \Cref{sec:exp3}. Here, we use the kernel pre-pruning.
\begin{figure}[ht]
    \centering
    \includegraphics[scale=0.20]{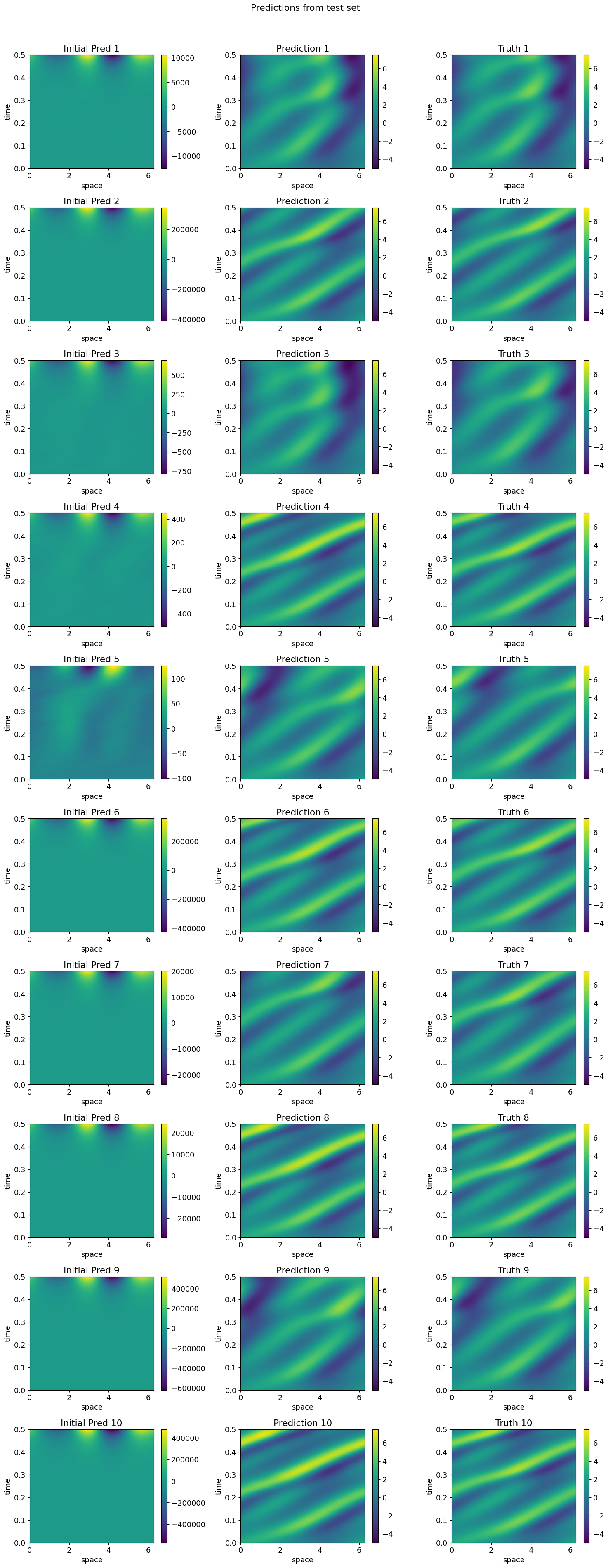}
    \caption{Approximating the trajectories of the Kuramoto-Sivashinsky PDE for all ten trajectories in the test set. On the left is the approximated trajectories using the kernel with initialization parameters, in the middle is the approximated trajectories from kEDMD after dictionary learning, and on the right is the true trajectories.}
    \label{fig:kse_test_full}
\end{figure}
In addition, in \Cref{fig:kse_test_full_pruned}, we include all ten test trajectories after performing the pruning of the kernel, as described in \Cref{sec:exp3}.
\begin{figure}[ht]
    \centering
    \includegraphics[scale=0.20]{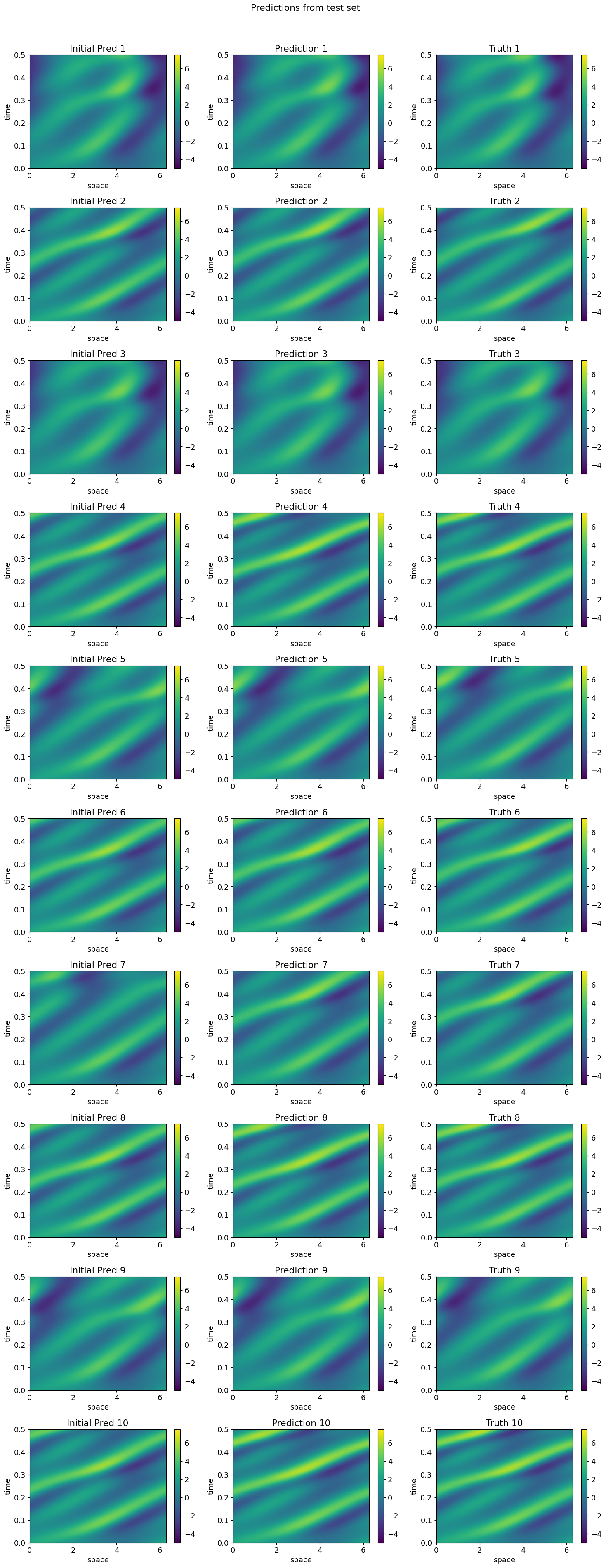}
    \caption{Approximating the trajectories of the Kuramoto-Sivashinsky PDE for all ten trajectories in the test set after pruning the kernel. On the left is the approximated trajectories using the kernel with initialization parameters, in the middle is the approximated trajectories from kEDMD after dictionary learning, and on the right is the  true trajectories.}
    \label{fig:kse_test_full_pruned}
\end{figure}
Finally, in \Cref{fig:kse_train_full_pruned}, we include ten trajectories from the training set. The plot shows two of the five diverging trajectories found in the training set (plot number 1 and 10 from the top). No such diverging was found in the test set; hence, we add ten trajectories from the training set as well.
\begin{figure}[ht]
    \centering
    \includegraphics[scale=0.20]{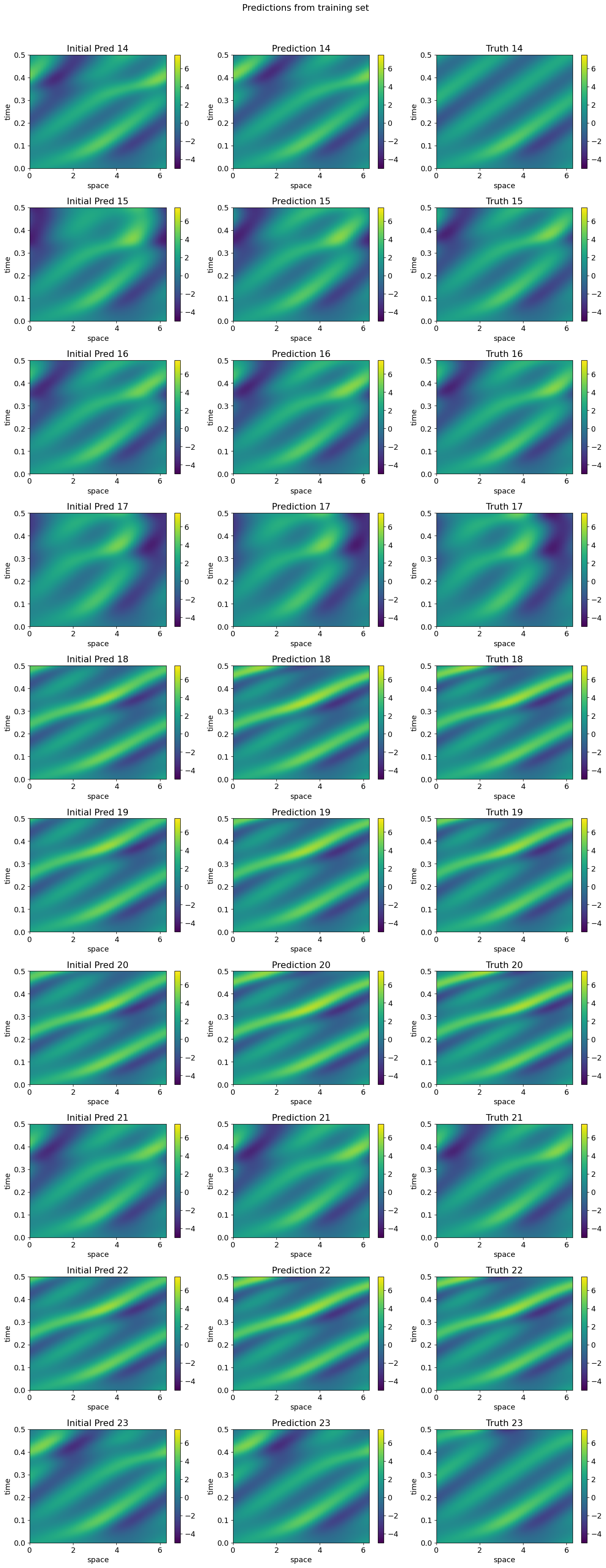}
    \caption{Approximating the trajectories of the Kuramoto-Sivashinsky PDE for ten trajectories in the training set after pruning the kernel. On the left is the approximated trajectory using the kernel with initialization parameters, in the middle is the approximated trajectory from kEDMD after dictionary learning, and on the right is the true trajectory.}
    \label{fig:kse_train_full_pruned}
\end{figure}
When increasing \(\Tilde N\) to 500, retrain the pruned kernel on this larger set \((\Tilde X, \Tilde Y)\), we find that the issues disappear. See \Cref{fig:kse_train_full_pruned_large} for the same initial conditions as in \Cref{fig:kse_train_full_pruned}.
\begin{figure}[ht]
    \centering
    \includegraphics[scale=0.20]{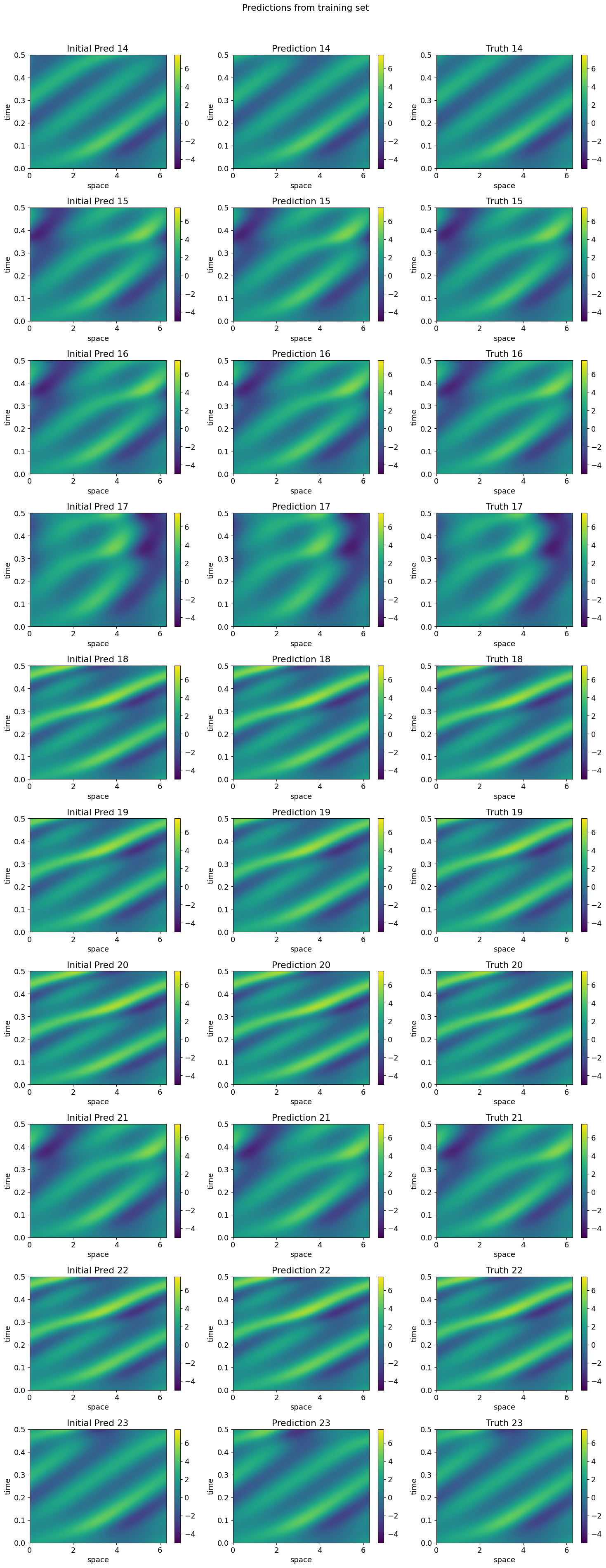}
    \caption{Approximating the trajectories of the Kuramoto-Sivashinsky PDE for ten trajectories in the training set after pruning the kernel, for \(\Tilde N = 500\). On the left is the approximated trajectory using the kernel with initialization parameters, in the middle is the approximated trajectory from kEDMD after fine-tuning the pruned kernel on the bigger set \((\Tilde X, \Tilde Y)\), and on the right is the true trajectory.}
    \label{fig:kse_train_full_pruned_large}
\end{figure}

\end{document}